\documentclass[11pt,a4paper]{amsart}

\usepackage[english]{babel} 
\usepackage[T1]{fontenc} 
\usepackage{amsmath}
\usepackage{amssymb} 
\usepackage{ucs} 
\usepackage[utf8x]{inputenc}
\usepackage[mathcal]{eucal} 
\usepackage[]{graphicx} 
\usepackage{psfrag}
\usepackage{latexsym}
\usepackage{amsthm} 
\usepackage{stmaryrd}
\usepackage{verbatim}
\usepackage{hyperref}

\usepackage{pgf,tikz}
\usetikzlibrary{arrows}
\usetikzlibrary[patterns]

\newtheorem{defi}{Definition}[section] 
\newtheorem{theo}[defi]{Theorem}
\newtheorem{coro}[defi]{Corollary} 

\newtheorem{lemma}[defi]{Lemma}
\newtheorem{prop}[defi]{Proposition}

\theoremstyle{remark} \newtheorem*{rem}{Remark}

\newcommand{\Ss}{\mathbb{S}}

\newcommand{\R}{\mathbb{R}} 
\newcommand{\Q}{\mathbb{Q}}
 
\newcommand{\Z}{\mathbb{Z}}
\newcommand{\proj}{\mathbb{P}}
\newcommand{\h}{\mathbb{H}}
\newcommand{\N}{\mathbb{N}} 

\newcommand{\e}{\varepsilon}
\newcommand{\p}{\varphi}

\newcommand{\Diff}{\mathrm{Diff}}
\newcommand{\Homeo}{\mathrm{Homeo}}

\newcommand{\suite}{_{n\in \N}}

\newcommand{\Stab}{\mathrm{Stab}}
\newcommand{\PSL}{\mathrm{PSL}}

\newcommand{\Gr}{\mathrm{Gr}}
\newcommand{\Log}{\mathrm{Log}}
\newcommand{\Aff}{\mathrm{Aff}}

\newcommand{\intoo}[2]{\mathopen{]}#1\,,#2\mathclose{[}}
\newcommand{\intff}[2]{\mathopen{[}#1\,,#2\mathclose{]}}
\newcommand{\intof}[2]{\mathopen{]}#1\,,#2\mathclose{]}}
\newcommand{\intfo}[2]{\mathopen{[}#1\,,#2\mathclose{[}}

\renewcommand{\tilde}{\widetilde}

\title{Convergence groups and semi conjugacy} 
\author{Daniel Monclair}
\date{\today}
\thanks{Partially supported by  ANR project GR-Analysis-Geometry (ANR-2011-BS01-003-02)}

\begin{document}

\maketitle

\begin{abstract} We study a simple problem that arises from the study of Lorentz surfaces and Anosov flows. For a non decreasing map of degree one $h:\mathbb{S}^1\to \mathbb{S}^1$, we are interested in groups of circle diffeomorphisms that act on the complement of the graph of $h$ in $\mathbb{S}^1\times \mathbb{S}^1$ by preserving a volume form. We show that such groups are semi conjugate to subgroups of $\mathrm{PSL}(2,\mathbb{R})$, and that when $h\in \mathrm{Homeo}(\mathbb{S}^1)$, we have a  topological conjugacy.  We also construct examples, where $h$ is not continuous, for which there is no such conjugacy.
\end{abstract}

\setcounter{tocdepth}{1}
\tableofcontents

\section{Introduction}
\subsection{Semi conjugacy} We say that a map $h:\Ss^1\to \Ss^1$ is  \textbf{non decreasing of degree one} if it is non constant and it admits a non decreasing lift $\tilde h:\R\to\R$ such that $\tilde h(x+1)=\tilde h(x)+1$ for all $x\in \R$.  Two representations $\rho_1,\rho_2 : \Gamma \to \Homeo(\Ss^1)$ are \textbf{semi conjugate} if there is a non decreasing map of degree one $h:\Ss^1\to \Ss^1$ such that $\rho_2(\gamma)\circ h=h\circ \rho_1(\gamma)$ for all $\gamma\in \Gamma$. We will say that $(\rho_1,\rho_2,h)$ is a \textbf{semi conjugate triple}.\\
\indent It was shown by Ghys in \cite{Gh87b} that semi conjugacy is an equivalence relation: if $(\rho_1,\rho_2,h)$ is a semi conjugate triple, then there is $h^*:\Ss^1\to \Ss^1$ non decreasing of degree one such that $(\rho_2,\rho_1,h^*)$ is a semi conjugate triple.\\
\indent If $h:\Ss^1\to\Ss^1$ is non decreasing of degree one, then we denote by $G(h)\subset \Ss^1\times \Ss^1$ the union of the graph of $h$ and the vertical segments joining discontinuities. If $(\rho_1,\rho_2,h)$ is a semi conjugate triple, then the map $(x,y)\mapsto (\rho_1(\gamma)(x),\rho_2(\gamma)(y))$ preserves $G(h)$ for all $\gamma\in \Gamma$.\\
\indent Values at discontinuities of a semi conjugacy are not particularly interesting. If we change $h$ by the function that is left continuous (or right continuous) and equal to $h$ except eventually at discontinuity points, then we still have a semi conjugacy. The important object in a semi conjugate triple $(\rho_1,\rho_2,h)$ is not the map $h$, but its "graph" $G(h)$.
\subsubsection{Collapsing non wandering intervals} When considering a representation $\rho:\Gamma \to \Homeo(\Ss^1)$ that preserves a Cantor set $K\subset \Ss^1$, such that orbits of points in $K$ are dense in $K$ (called an exceptional minimal set), it is standard to consider the collapsed action $\rm{Col}(\rho)$ defined by collapsing the connected components of $\Ss^1\setminus K$ to points. More precisely, we can consider a continuous map $h:\Ss^1\to \Ss^1$ that is non decreasing of degree one, such that the intervals where $h$ is constant are exactly the connected components of $\Ss^1\setminus K$. It induces a unique representation $\rm{Col}(\rho) :\Gamma \to \Homeo(\Ss^1)$ such that $(\rho,\rm{Col}(\rho),h)$ is a semi conjugate triple. The interest in considering this collapsed action is that the orbits of $\rm{Col}(\rho)$ are dense in $\Ss^1$. It is well defined up to a conjugacy in $\Homeo(\Ss^1)$.\\
\indent If $\rho$ has values in $\Diff(\Ss^1)$, it is not clear wether $\rm{Col}(\rho)$ can also be asked to be differentiable. However, if $\rho$ has values in $\PSL(2,\R)$ and acts projectively on $\Ss^1=\R\proj^1$, then the collapsed action $\rm{Col}(\p)$ still has the convergence property, i.e. all sequences have north/south dynamics. Using a theorem of Gabai and Casson-Jungreis (\cite{Gabai},\cite{CJ}), this implies that $\rm{Col}(\rho)$ is topologically conjugate to the projective action of a subgroup of $\PSL(2,\R)$ (we will give more details on convergence groups in section \ref{sec:circle_dynamics}). More precisely, the class of \textbf{topologically Fuchsian} representations, i.e. representations $\rho:\Gamma \to \Homeo(\Ss^1)$ that are topologically conjugate to the projective action of a subgroup of $\PSL(2,\R)$, is stable under collapsing.
\subsubsection{Opening orbits} It is also possible to go backwards. Given a representation $\rho : \Gamma \to \Homeo(\Ss^1)$, and a point $x_0\in \Ss^1$, we can open the orbit of $x$ to intervals on which the action of $\Stab(x_0)$ can be chosen. Indeed, choose an increasing map of degree one $h:\Ss^1\to \Ss^1$ whose points of discontinuity are exactly the points of the orbit of $x$. For every $x \in \Gamma. x_0$, let $I_x$ be the interval between the left and right limits of $h$ at $x$. Fix an action $\alpha : \Stab(x_0) \to \Homeo(I_{x_0})$. For every $x\in \Gamma.x_0$, we choose $\delta_x$ such that $x=\rho(\delta_x)(x_0)$ and a homeomorphism $h_x : I_{x_0} \to I_x$. We define $\hat \rho : \Gamma \to \Homeo(\Ss^1)$ such that: \begin{enumerate} \item $\hat\rho(\gamma)(h(x))=h(\rho(\gamma)(x))$ if $x\notin \Gamma .x_0$ \item  $\hat \rho (\gamma) (y) = \alpha (\gamma)(y)$ if $y\in I_{x_0}$ and $\gamma \in \Stab(x_0)$ 
\item $\hat \rho(\gamma)(y) = h_{\rho(\gamma)(x)}(\alpha(\delta_{\rho(\gamma)(x)}^{-1}\gamma \delta_x)(h_x^{-1}(y)))$ for $y\in I_x$ and $x\in \Gamma.x_0$  \end{enumerate}

\indent This defines an action $\hat \rho : \Gamma \to \Homeo(\Ss^1)$ such that $(\rho,\hat \rho,h)$ is a semi conjugate triple.\\
\indent Note that if $\rho$ has dense orbits, then $\rm{Col}(\hat \rho)=\rho$ (because the collapsed action is unique, but there are different possibilities for opening orbits).\\
\indent Once again, if $\rho$ has values in $\Diff(\Ss^1)$, then it is not easy to construct $\hat \rho$ in a differentiable way. A well known construction of Denjoy (\cite{Denjoy}) shows that if we start with an irrational rotation, then we can open an orbit in a $C^1$ way. However, Denjoy also showed that it is not possible to obtain a $C^2$ diffeomorphism. A theorem of Matsumoto (see \cite{Matsumoto}) states that if $\Gamma_g$ is the fundamental group of the closed oriented surface $\Sigma_g$ of genus $g\ge 2$, and $\rho:\Gamma_g\to\PSL(2,\R)$ is the the representation given by a hyperbolic metric on $\Sigma_g$, then it is not possible to open an orbit with a $C^2$ action (because the Euler number stays the same).

\subsection{Semi conjugacy and area preserving actions}
Given a semi conjugate triple $(\rho_1,\rho_2,h)$, the underlying group $\Gamma$ acts on $M_h=\Ss^1\times \Ss^1\setminus G(h)$ through the maps $(x,y)\mapsto (\rho_1(\gamma)(x),\rho_2(\gamma)(y))$ for $\gamma \in \Gamma$. We will say that a semi conjugate triple $(\rho_1,\rho_2,h)$ is \textbf{area preserving} if there is a continuous volume form $\omega$ on $M_h$ preserved by the action of $\Gamma$.\\
\indent  Notice that this implies that $\rho_1$ and $\rho_2$ are differentiable (they take values in $\Diff^{k+1}(\Ss^1)$ if $\omega$ is $C^k$, $k\ge 0$). If we write $\omega=\omega(x,y)dx\wedge dy$, then the fact that the triple $(\rho_1,\rho_2,h)$ preserves $\omega$ is equivalent to the following formula:
$$\forall \gamma\in \Gamma~~ \forall (x,y)\in M_h~~  \rho_1(\gamma)'(x)\rho_2(\gamma)'(y)=\frac{\omega(x,y)}{\omega(\rho_1(\gamma)(x),\rho_2(\gamma)(y))} $$
\indent In \cite{Monclaira}, we studied the case where $h=Id$ (hence $\rho_1=\rho_2$). The main example is given by the projective action of $\PSL(2,\R)$ on $\Ss^1\approx \R\proj^1$ and the volume form $\frac{4}{(x-y)^2}dx\wedge dy$. We proved that if $(\rho,\rho,Id)$ is area preserving, then $\rho$ is topologically Fuchsian. We wish to extend this result to the general case of semi conjugacy. The original motivation for this problem is the study of Lorentz surfaces.
\subsubsection{Lorentzian interpretation} If the semi conjugate triple $(\rho_1,\rho_2,h)$ preserves a volume form $\omega(x,y)dx\wedge dy$ on $M_h$, then it also preserves the pseudo Riemannian metric $\omega(x,y)dxdy$, which is Lorentzian, i.e. of signature $(1,1)$. Notice that for any semi conjugate triple $(\rho_1,\rho_2,h)$ (i.e. not necessarily area preserving), the action on $M_h$ preserves the conformal class of the flat metric $dxdy$. However, preserving a Lorentzian conformal class in dimension two is not a rigid notion (the conformal group may be infinite dimensional). The isometry group of a Lorentz surface is always a Lie group of dimension at most $3$ (not necessarily connected), and we will show that only subgroups of $\PSL(2,\R)$ can appear for metrics in the conformal class $[dxdy]$ on $M_h$.\\
\indent The conformal classes $(M_h,[dxdy])$ share an important property: they are globally hyperbolic (this roughly means that the space of isotropic lines is a smooth manifold). Moreover, they are spatially compact (isotropic lines form a compact manifold). The study of isometry groups of spatially compact Lorentz surfaces was started in \cite{Monclairb}.\\
\indent The problem can be formulated in terms of $h$ and $\omega$: we denote by $G_\omega$ the  group of orientation and time orientation preserving (i.e. preserving the set of tangent vectors $(u,v)$ such that  $u<0$ and $v>0$)  isometries of the Lorentz metric associated to $\omega$. If $\p\in G_\omega$, then there are $f,g\in \Diff(\Ss^1)$ such that $\p(x,y)=(f(x),g(y))$ for all $(x,y)\in M_h$ (because $\p$ preserves the horizontal and vertical lines which are the isotropic directions). We will set $\rho_1(\p)=f$ and $\rho_2(\p)=g$. This defines two representations $\rho_1,\rho_2:G_\omega \to \Diff(\Ss^1)$ such that the semi conjugate triple $(\rho_1,\rho_2,h)$ preserves the volume form $\omega$.\\
\indent However, no Lorentzian background will be required in this paper, as we will deal with metrics that are only continuous, and all the classical tools (curvature, geodesics, exponential map, \dots) are only defined when the metric is at least $C^2$.
\subsubsection{Results} As mentioned earlier, if $(\rho,\rho,Id)$ is area preserving, then $\rho$ is topologically conjugate to a representation in $\PSL(2,\R)$. This result can be extended to the more general setting of topological conjugacy:
\begin{theo} \label{main_conjugacy} Let $(\rho_1,\rho_2,h)$ be a semi conjugate triple where $h\in \Homeo(\Ss^1)$. If $(\rho_1,\rho_2,h)$ is area preserving, then $\rho_1$ is topologically Fuchsian. \end{theo}
This can be slightly reformulated in terms of a actions on the circle: if $\rho_1:\Gamma\to\Diff(\Ss^1)$ is such that are $\rho_2$ and $h\in \Homeo(\Ss^1)$ forming an area preserving triple $(\rho_1,\rho_2,h)$, then $\rho_1$ is topologically Fuchsian.\\
\indent In the general case, we obtain a semi conjugacy:
\begin{theo} \label{main_semi_conjugacy} Let $(\rho_1,\rho_2,h)$ be an area preserving triple. Then $\rho_1$ is semi conjugate to a representation in $\PSL(2,\R)$. \end{theo}
This result is  almost a particular case of Theorem 1.2 in \cite{Monclairb},  with the exception that we no longer assume that the preserved volume form is smooth.

\subsection{A new class of "semi convergence" groups} A natural question is to ask if the result of  Theorem \ref{main_conjugacy} is true under the assumptions of Theorem \ref{main_semi_conjugacy}.  In section \ref{sec:elements}, we will describe elements of $G_\omega$ and  observe that there are examples where an element $\rho_1(\p)$ has three fixed points yet is not trivial, therefore it cannot be  topologically conjugate to an element of $\PSL(2,\R)$. More importantly, we will see that these examples can occur in a situation where $G_\omega$ contains a free group.

\begin{theo} \label{free_group_example} There are faithful representations $\rho_1,\rho_2 : \mathbb F_3 \to \Diff(\Ss^1)$ and $h:\Ss^1\to \Ss^1$ non decreasing of degree one such that $(\rho_1,\rho_2,h)$ is area preserving, and such that $\rho_2(\mathbb F_3)$ contains non trivial elements with three fixed points. \end{theo}

Such examples will be obtained by considering a representation $\rho_1:\mathbb F_3 \to \PSL(2,\R)$ associated to a certain hyperbolic metric on the twice punctured torus (whose fundamental group is $\mathbb F_3$), and perturbing the representation in $\Diff(\Ss^1)$ by adding a third fixed point to one of the generators. The action $\rho_2$ can be seen as a different opening of an orbit of the collapsed representation $\rm{Col}(\rho_1)$.

\subsection{Transverse structure of Anosov flows}  Let  $(M, \p^t)$ be an Anosov flow on a compact 3-manifold. It is known that the quotient space $Q^{\p}$ of lifts of orbits to the universal cover is diffeomorphic to  $\R^2$ (see \cite{B95}).  The  stable and unstable foliations determine two transversal one-dimensional foliations of $Q^{\p}$, hence a conformal Lorentz structure (whose isotropic lines are the leaves of these foliations).\\
\indent The fundamental group $\pi_1(M)$ acts  naturally on $Q^{\p}$ by preserving this structure. In fact, Anosov flows up to topological equivalence tend to be classified by this conformal action of the fundamental group (\cite{B95}).\\
\indent If  $\p^t$ preserves a volume form, then we get a Lorentz metric on $Q^{\p}$ which is in general $C^{1+ \rm{Zygmund}} $ (and thus $C^{1+\alpha}$, for any $0< \alpha <1$). Higher regularity ($C^2$, or even $C^{1+\rm{zygmund}}$) implies not only rigidity for the Lorentz space $Q^{\p}$ (the curvature is constant), but also rigidity for the flow itself (smooth conjugacy with a reparametrisation of an algebraic flow, \cite{HK} and \cite{Gh87a}). The fundamental group acts isometrically for this Lorentz metric. \\ 
\indent  Examples of such flows are abundant, and exist for instance on many hyperbolic manifolds (see \cite{FH}). Theses examples  belong to a special subcategory of Anosov flows: they are $\R$-covered (\cite{B01}). It means that the quotient space of $Q^{\p}$ by one of the foliations  defined by the stable or unstable foliation (and automatically by both), is Hausdorff,  hence homeomorphic to $\R$. This is equivalent to saying that the Lorentz space $Q^{\p}$ is globally hyperbolic. \\
\indent In these examples, the Lorentz metric is defined on a subset of $\R\times \R$, and the action of $\pi_1(M)$ gives diffeomorphisms of the real line. However, on some examples, circle diffeomorphisms arise, mostly for geodesic flows on negatively curved surfaces, and for the action of the fundamental group of a Seifert piece in a graph manifold (see \cite{Ba96}).

\subsection{Overview} We will start by recalling the necessary notions of circle dynamics, namely convergence groups, minimal invariant sets, and semi conjugacy. In section \ref{sec:convergence}, we will give the proofs of Theorem \ref{main_semi_conjugacy} and Theorem \ref{main_conjugacy}. Section \ref{sec:elements} gives a description of the possible elements of $G_\omega$, and section \ref{sec:examples} is devoted to the proof of Theorem \ref{free_group_example}.

\section{Circle dynamics}  \label{sec:circle_dynamics}
\subsection{Convergence groups} Finding a conjugacy between a subgroup of $\Homeo(\Ss^1)$  and a subgroup of $\PSL(2,\R)$, when it exists, is a rather complicated exercise. But there is a characterisation of the existence of such a conjugacy that does not require to find it explicitly.

\begin{defi} A sequence $(f_n)\suite \in \Homeo(\Ss^1)^{\N}$ has the convergence property if there are $a,b\in \Ss^1$ such that, up to a subsequence,  $f_n(x)\to b$ for all $x\ne a$.\\ A group $G \subset \Homeo(\Ss^1)$ is a convergence group if every  sequence in $G$ either satisfies the convergence property or has an equicontinuous subsequence.  \end{defi}
The classical definition of a convergence group also implies the sequence of the inverses $f_n^{-1}$ and locally uniform convergence, but it is not necessary in the case of $\Ss^1$. \\
\indent We say that a group $G\subset \Homeo(\Ss^1)$ is \textbf{topologically Fuchsian} if there is $h\in \Homeo(\Ss^1)$ such that $h^{-1} G h \subset \PSL(2,\R)$. Note that we do not ask for $G$ to be discrete. Subgroups of $\PSL(2,\R)$ all satisfy the convergence property, and since this property is invariant under topological conjugacy, so do topologically Fuchsian groups. This happens to be an equivalence.
\begin{theo}[Gabai \cite{Gabai}, Casson-Jungreis \cite{CJ}] \label{convergence_groups} A convergence group $G \subset \Homeo(\Ss^1)$ is topologically Fuchsian. \end{theo}

\subsection{Minimal invariant sets} An important object in the study of  groups of circle homeomorphisms is a minimal closed invariant set.  Given a group $G \subset \Homeo(\Ss^1)$,  exactly one of the following conditions is satisfied (see \cite{Gh01} for a proof and more detail):
\begin{enumerate} \item $G$ has a finite orbit \item All orbits of $G$ are dense \item There is a compact $G$-invariant subset $K\subset \Ss^1$ which is infinite and different from $\Ss^1$, such that the orbits of points of $K$ are dense in $K$. \end{enumerate}

\indent In the first case, all finite orbits have the same cardinality.  In the third case, the set $K$ is unique, and it is homeomorphic to a Cantor set. We will call a group $G\subset \Homeo(\Ss^1)$ \textbf{non elementary} if it does not have any finite orbit, and use $L_{G}$ to denote $\Ss^1$ in the second case and  the $G$-invariant compact set $K$ in the third case. 

\subsection{Semi conjugacy} Recall that $h:\Ss^1\to \Ss^1$ is non decreasing of degree one if it is non constant and it has a lift  $\tilde h:\R\to\R$ such that $\tilde h(x+1)=\tilde h(x)+1$ for all $x\in \R$. We say that $\rho_1,\rho_2:\Gamma \to \Homeo(\Ss^1)$ are semi conjugate, or that $(\rho_1,\rho_2,h)$ is a semi conjugate triple, if $\rho_2(\gamma)\circ h=h\circ \rho_1(\gamma)$ for all $\gamma \in \Gamma$.\\
\indent If $a,b\in \Ss^1$, the we denote by $\intoo{a}{b}$ the set $\{ x\in \Ss^1 \vert a<x<b\le a\}$ (and define $\intff{a}{b}$, $\intfo{a}{b}$, $\intof{a}{b}$ in a similar way). If $h:\Ss^1\to \Ss^1$ is non decreasing of degree one, then we denote by $h_l$ (resp. $h_r$) the map that is left (resp. right) continuous and equal to $h$ except at points where $h$ is not left (resp. right) continuous. We define $G(h)\subset \Ss^1\times \Ss^1$ as the union of the segments $\{x\} \times \intff{h_l(x)}{h_r(x)}$ for all $x\in \Ss^1$. Recall that the set of discontinuity points of $h$ is at most countable, so $\intff{h_l(x)}{h_r(x)}=\{ h(x)\}$ except on a countable set.\\
\indent Let us recall a few results of \cite{Gh87b} on semi conjugacy. First, as mentioned in the introduction, the advantage of the definition of semi conjugacy that we use (contrary to the standard definition where $h$ is asked to be continuous) is that semi conjugacy is an equivalence relation.\\
\indent If $\Gamma=\Z$, then $f,g\in \Homeo(\Ss^1)$ are semi conjugate if and only if they have the same rotation number (the rotation number of $f\in \Homeo(\Ss^1)$ can actually be defined as the unique $\alpha\in \Ss^1$ such that $f$ is semi conjugate to the rotation $R_\alpha$).\\
\indent For elementary groups, we have a simple characterisation of semi conjugacy:
\begin{prop} \label{finite_orbit_semi_conjugacy} Let $\rho :\Gamma \to \Homeo(\Ss^1)$ have a finite orbit $E \subset \Ss^1$ with at least two elements. A representation $\tau:\Gamma \to \Homeo(\Ss^1)$ is semi conjugate to $\rho$ if and only if it has a finite orbit $F\subset \Ss^1$ such that there is a cyclic order preserving bijection from $E$ to $F$ which is equivariant under the actions of $\Gamma$. \end{prop}
Consequently, if $\rho$ and $\tau$ are semi conjugate, then $\rho$ is  elementary if and only if $\tau$ is  elementary. If $\omega$ is a continuous volume form on $M_h$, and if $G\subset G_\omega$ is a subgroup, then $\rho_1(G)$ is elementary if and only if $\rho_2(G)$ is elementary.
\begin{defi} Let $\omega$ be a continuous volume form on $M_h$. We say that a subgroup $G\subset G_\omega$ is elementary if $\rho_1(G)$ is elementary. \end{defi}


\section{Convergence property for $\rho_1$ and $\rho_2$}
 \label{sec:convergence}
\subsection{Faithfulness of the actions}
It is a classical result in pseudo Riemannian geometry that an isometry with a fixed point where the derivative is the identity must be the identity everywhere (provided the manifold is connected). In other words, isometries are defined by their $1$-jet at any point. This is a crucial fact, since it is the very definition of a rigid geometric structure (in the sense of Gromov). However, the standard proof uses the local existence and uniqueness of geodesics, which is not true  for continuous metrics. We will see that this result remains true in our case, by using the fact that we can still define isotropic geodesics and that they are sufficient in order to get a linearisation. 
\begin{lemma} \label{continuous_rigidity} Let $\omega$ be a continuous volume form on $M_h=\Ss^1\times \Ss^1 \setminus G(h)$. If $\p \in G_\omega\setminus \{Id\}$ fixes a point $(x_0,y_0)\in M_h$, then $\rho_1(\p)'(x_0)\rho_2(\p)'(y_0)=1$ and $\rho_1(\p)'(x_0)\ne 1$.
\end{lemma}

\begin{proof} Write $f=\rho_1(\p)$ and $g=\rho_2(\p)$. \\ \indent The identity $\omega(f(x),g(y))f'(x)g'(y)=\omega(x,y)$ considered  at $(x_0,y_0)$ shows that $f'(x_0)g'(y_0)=1$. Assume that  $f'(x_0)=1$ (hence $g'(y_0)=1$). Since $\p\ne Id$, let us assume that $f\ne Id$ (the case where $g\ne Id$ is similar).\\
\indent Let $x(t)$ be a maximal solution of the Cauchy problem:
\begin{equation*}
\left\{
\begin{array}{cc}
  x'(t)  = & \frac{1}{\omega(x(t),y_0)}   \\
 x(0) = & x_0   \\
\end{array}
\right.
\end{equation*}
 \indent Not only does $x$ exist (Cauchy-Peano Theorem), but it is also unique (so are solutions to all equations $y'=F(y)$ in $\R$ where $F>0$). It should be seen as a parametrisation of the geodesic $(\Ss^1\times \{y_0\})\cap M_h$ for the pseudo Riemannian metric associated to $\omega$. Since $x'>0$, it is a diffeomorphism from an open interval $I\subset \R$ onto its image $\Ss^1\setminus K$ where $K$ is the set of points $x\in \Ss^1$ such that $(x,y_0)\in G(h)$. Let $\alpha = x^{-1} \circ f \circ x$. A simple calculation shows that $\alpha'(t)=1$ for all $t\in I$. Since $\alpha(0)=0$, we see that $\alpha =Id$ and $f(x)=x$ for all $x$ such that $(x,y_0)\in M_h$.\\
\indent Let $K=\intff{a_1}{a_2}$  and let $\intff{b_1}{b_2}=\intff{h_l(x_0)}{h_r(x_0)}$. If $a_1=a_2$, then $f$ is the identity on a dense subset of $\Ss^1$, so $f=Id$. If $a_1\ne a_2$, then $(a_1,b_1)\in M_h$ is a fixed point of $\p$ such that $f'(a_1)=1$, so $f$ is the identity on all points $x$ such that $(x,b_1)\in M_h$, which includes $\intff{a_1}{a_2}$, so $f=Id$.

\end{proof}
The first consequence of this rigidity is the faithfulness of the actions $\rho_1$ and $\rho_2$.
\begin{coro} Let $\omega$ be a continuous volume form on $M_h$. The representations $\rho_1, \rho_2:G_\omega \to \Diff(\Ss^1)$ are faithful. \end{coro}

\begin{proof} Assume that $\rho_1(\p)=Id$. Then $g=\rho_2(\p)$ satisfies $g\circ h=h$, which implies that $g$ has fixed points. If $g(y)=y$, then choose $x\in \Ss^1$ such that $(x,y)\notin G(h)$. Then $(x,y)$ is a fixed point of $\p$ such that $f'(x)=1$. Lemma \ref{continuous_rigidity} implies that $\p=Id$. \end{proof}

\subsection{Elementary groups} We will now give a proof of Theorem \ref{main_semi_conjugacy} for elementary groups. We start with stabilizers of points.\\
\indent Recall that the affine group $\Aff(\R)$ can be realised as a subgroup of $\PSL(2,\R)$ as the stabilizer of a point in $\Ss^1$.
\begin{lemma} \label{stabilizer} Let $h$ be an non decreasing map of degree one of $\Ss^1$, and let $\omega$ be a continuous volume form on $M_h=\Ss^1\times \Ss^1\setminus G(h)$. Let $x_0\in \Ss^1$, and set  $G=\{ \p \in G_\omega \vert \rho_1(\p)(x_0)=x_0\}$. There is a  representation $\rho:G\to \Aff(\R)\subset \PSL(2,\R)$ that  is semi conjugate to $\rho_1$. \end{lemma}

\begin{proof}   Let us assume that $G$ is non trivial.\\
\indent  Let $b=h(x_0)$. If $\p \in G$, then $\rho_2(\p)(b)=b$, i.e. $\p$ preserves the horizontal line $\Ss^1\times \{b\}$. Fix $a\in \Ss^1$ such that $(a,b)\in M_h$, and let $x$ be a maximal solution  of the Cauchy problem:
\begin{equation*}
\left\{
\begin{array}{cc}
  x'(t)  = & \frac{1}{\omega(x(t),b)}   \\
 x(0) = & a   \\
\end{array}
\right.
\end{equation*}
\indent Just as in Lemma \ref{continuous_rigidity}, it is a parametrisation of the horizontal geodesic passing through $(a,b)$, and it is a  diffeomorphism from an open interval $I\subset \R$ onto $\Ss^1\setminus \overline{h^{-1}(\{b\})}$. If $\p \in G$, then a simple calculation shows that $(x^{-1} \circ \rho_1(\p) \circ x)'(t) = \frac{1}{\rho_2(\p)'(b)}$ for all $t\in \R$ such that  $x(t)$ is defined. This shows that $x$ conjugates the action of $G$ on $\Ss^1\setminus \overline{h^{-1}(\{b\})}$ with a subgroup of $\Aff(\R)$.\\
\indent Since $G$ is non trivial, the interval $I$ has non trivial affine diffeomorphisms, so $I$ is either $\R$, either affinely equivalent to $\intoo{-\infty}{0}$ or $\intoo{0}{+\infty}$, in which case the action of the affine group is differentially conjugate to the action of $\R$ on itself by translations. Therefore, up to changing $I$ and $x$ (while preserving the affine structure), we can assume that $I=\R$.  \\
\indent Let $\psi : \Ss^1\to \Ss^1=\R\cup \{\infty\}$ be defined by $\psi =x^{-1}$ on $\Ss^1\setminus \overline{h^{-1}(\{b\})}$ and $\psi\equiv \infty$ on $\overline{h^{-1}(\{b\})}$. It provides a semi conjugacy between $\rho_1(G)$ and a representation  $\rho:G\to \Aff(\R)$.\\
\end{proof}

\begin{prop} \label{elementary_semi_conjugacy} Let $h$ be an non decreasing map of degree one of $\Ss^1$, and let $\omega$ be a continuous volume form on $M_h=\Ss^1\times \Ss^1\setminus G(h)$.  Assume that $G\subset G_\omega$ is elementary. Then there is a  representation $\rho : G\to \PSL(2,\R)$ that is semi conjugate to $\rho_1$. 
\end{prop}

\begin{proof}  Let $L_1\subset \Ss^1$ be a finite orbit for $\rho_1(G)$. If $\sharp L_1 =1$, then Lemma \ref{stabilizer} applies. If $\sharp L_1=k\ge 2$, then let $L_1=\{ x_{\overline 1},\dots ,x_{\overline k}\}$ (where the indices are in $\Z/k\Z$, and $x_{\overline 1} <\cdots <x_{\overline k} < x_{\overline 1}$). Since elements of $\rho_1(G)$ preserve the cyclic order, there is a morphism $\sigma : G\to \Z/k\Z$ such that $\rho_1(\p)(x_{ i})=x_{ i+\sigma(\p)}$ for all $i\in \Z/k\Z$ and $\p \in G$. Since $G$ acts transitively on $L_1$, we necessarily have $\sigma(G)=\Z/k\Z$. Then $\p \mapsto R_{\frac{\sigma(\p)}{k}}$ is a representation of $G$ in $\mathrm{SO}(2,\R)\subset \PSL(2,\R)$ that is semi conjugate to $\rho_1$ by  Proposition \ref{finite_orbit_semi_conjugacy}.
\end{proof}

Note that the reason why we had to start with stabilizers of points is that the semi conjugacy defined in the proof of Proposition \ref{elementary_semi_conjugacy} is a constant map in the case where $\sharp L_1=1$, hence does not satisfy our definition of a non decreasing map of degree one.

\subsection{The general case} Let us reformulate Theorem \ref{main_semi_conjugacy} in terms of the group $G_\omega$.

\begin{theo} \label{semi_conjugacy} Let $h$ be an non decreasing map of degree one of $\Ss^1$, and let $\omega$ be a continuous volume form on $\Ss^1\times \Ss^1\setminus G(h)$. The group $\rho_1(G_\omega)$ is semi conjugate to a subgroup of $\PSL(2,\R)$.  \end{theo}

\begin{proof} If $h$ has a finite number of values, then $G_\omega$ is elementary, and we can apply Proposition \ref{elementary_semi_conjugacy}. We can now assume that $h$ is not finite valued. \\
\indent Let $U_1$ be the union of the  open intervals where $h$ is constant, and let $U_2$ be the reunion of open intervals between discontinuities of $h$. The complement of $U_1$ (resp. of $U_2$) is a closed $\rho_1$-invariant (resp. $\rho_2$-invariant) set.\\
\indent Let $p_1,p_2 : \Ss^1\to \Ss^1$ be continuous non decreasing maps of degree one such that the intervals where $p_i$ is constant are exactly the connected components of $U_i$. They induce representations $\hat \rho_1, \hat \rho_2 : G_\omega \to \Homeo(\Ss^1)$ such that $\hat \rho_i \circ p_i =p_i\circ \rho_i$, and we now have a homeomorphism $\hat h$ such that $\hat h\circ \hat \rho_1=\hat \rho_2\circ \hat h$. We are going to show that $\hat \rho_1$ and $\hat \rho_2$ are topologically Fuchsian, i.e. that they satisfy the convergence property.\\
\indent Let $(\p_n)\suite \in G_\omega^\N$ be a sequence such that $\hat \rho_1(\p_n)$ has no equicontinuous subsequence.\\
To simplify the notations, we will set $f_n=\rho_1(\p_n)$, $g_n=\rho_2(\p_n)$, $\hat f_n =\hat \rho_1(\p_n)$ and $\hat g_n=\hat \rho_2(\p_n)$.\\
\indent We are first going to show that the sequences $\hat f_n(x)$ have at most two distinct limit points. Indeed, assume that there are three distinct points $\hat \alpha < \hat \beta < \hat\gamma < \hat \alpha$ in $\Ss^1$ and $\hat a, \hat b,\hat c\in \Ss^1$ such that $\hat f_n(\hat a) \to \hat \alpha$, $\hat f_n(\hat b) \to \hat \beta$ and $\hat f_n (\hat c) \to \hat \gamma$. We consider a subsequence such that $f_n(a)\to \alpha$, $f_n(b)\to \beta$ and $f_n(c)\to \gamma$ for some lifts  $a,b,c,\alpha,\beta,\gamma$ with respect to $p_1$.\\
\indent We also set $\hat \alpha'=\hat h(\hat \alpha), \hat a'=\hat h(\hat a), \dots$ and choose lifts $a',b',c',\alpha',\beta',\gamma'$ with respect to $p_2$. \\
\indent First, let us assume that $g'_n(a')\to 0$. Let $K$ be a compact interval of $\Ss^1\setminus \{\alpha\}$ that contains $\beta$ and $\gamma$ in its interior. There is $n_0\in \N$ such that $f_n(x)\in K$ for all $x\in \intff{b}{c}$ and all $n\ge n_0$. This implies that the sequence $(f_n(x),g_n(a'))$ stays in a compact set of $M_h$, and:
$$f'_n(x)=\frac{1}{g'_n(a')} \frac{\omega(x,a')}{\omega(f_n(x),g_n(a'))} \to +\infty$$
\indent Fatou's Lemma implies that $\int_b^cf'_n(x)dx \to \infty$, which is impossible because $\int_{\Ss^1} f'_n(x)dx =1$. This shows that the sequence $g'_n(a')$ cannot converge to $0$, nor can any subsequence, and there is a constant $C>0$ such that $g'_n(a')\ge C$ for all $n\in \N$. Since $a$, $b$ and $c$ have similar roles, we can also assume that $g'_n(b')\ge C$ and $g'_n(c')\ge C$.\\
\indent We now see that $f'_n(x)=\frac{1}{g'_n(a')}\frac{\omega(x,a')}{\omega(f_n(x),g_n(a'))}$ is uniformly bounded on $\intff{b}{c}$, therefore $f_n$ is equicontinuous on this interval and up to a subsequence we can assume that $f_n$ converges uniformly on $\intff{b}{c}$ (Ascoli's Theorem). Since $b$ and $c$ have a similar role to $a$, there is a subsequence that converges uniformly on $\intff{c}{a}$ and on $\intff{a}{b}$, therefore on all of $\Ss^1$, which is impossible because we assumed that $\hat f_n$ has no equicontinuous subsequence.\\

\indent We now know that there are at most two possible limits for $f_n$, say $\hat \alpha$ and $\hat \beta$ (we keep similar notations for $\hat a$, $\hat a'$, $a$, $a'$, $b$ \dots). Let $A$ (resp. $B$) be the set of points $x\in \Ss^1$ such that $\hat f_n(x)\to \hat \alpha$ (resp. $\hat f_n(x)\to \hat \beta$).\\
\indent If $x,y\in A$, then one of the two intervals $\intff{\hat f_n(x)}{\hat f_n(y)}$ and  $\intff{\hat f_n(y)}{\hat f_n(x)}$ shrinks to $\{ \hat \alpha\}$. This implies that one of the intervals $\intff{x}{y}$ and $\intff{y}{x}$ is included in $A$, hence $A$ is connected, and it is an interval of $\Ss^1$. The same goes for $B$.\\
\indent Assume that neither $A$ nor $B$ is reduced to a point or void. Let $\hat y\in \hat h(\mathring B)$, and $y\in p_2^{-1}(\hat y)$.  First, assume that $g'_n(y)\to 0$. Then $f'_n(x)\to \infty$ for all $x\in p_1^{-1}(A)$, and $\int_{p_1^{-1}(A)} f'_n(x)dx \to \infty$, which is absurd. This shows that there is $C>0$ such that $g'_n(y)\ge C$ for all $n\in \N$. Then $f'_n$ is uniformly bounded on $p_1^{-1}(A)$, and the sequence $\hat f_n$ is equicontinuous on $A$. Similarly, the sequence $\hat g_n$ is equicontinuous on $\hat h(B)$, and $\hat f_n$ is equicontinuous on $B$. Since we can choose a subsequence such that $A\cup B=\Ss^1$ (by making $f_n$ convergent on a dense countable subset of $\Ss^1$), this implies that $\hat f_n$ is equicontinuous, which is absurd. Therefore $A$ or $B$ contains at most one point, and the sequence $\hat f_n$ satisfies the convergence property.
\end{proof}

As a corollary of the proof, we get the convergence property for $\rho_1$ and $\rho_2$ when $h\in \Homeo(\Ss^1)$, i.e. Theorem \ref{main_conjugacy}.

\begin{theo} Let $h\in \Homeo(\Ss^1)$, and let $\omega$ be a continuous volume form on $\Ss^1\times \Ss^1\setminus G(h)$. The group $\rho_1(G_\omega)$ is topologically Fuchsian. \end{theo} 
\begin{proof} If $h\in \Homeo(\Ss^1)$, then $U_1=\emptyset$ and $U_2=\emptyset$ in the proof of Theorem \ref{semi_conjugacy}. This implies that $\hat \rho_1 = \rho_1$, so $\rho_1(G_\omega)=\hat \rho_1(G_\omega)$ is topologically Fuchsian.
\end{proof}

\section{Elements of $G_\omega$} \label{sec:elements}

All elements of $\PSL(2,\R)$ can appear (because of the form $\omega_0=\frac{4dx\wedge dy}{(x-y)^2}$ on $M_{Id}=\Ss^1\times \Ss^1\setminus \Delta$ for which $\rho_1(G_{\omega_0})=\PSL(2,\R)$), so we are now going to focus on  elements of $\rho_1(G_\omega)$ that are not  conjugate in $\Homeo(\Ss^1)$ to elements in $\PSL(2,\R)$. We will construct three types of examples: first introducing a parabolic fixed point in a hyperbolic element of $\PSL(2,\R)$, then opening the fixed point of a parabolic element of $\PSL(2,\R)$, and finally considering the lift of a parabolic element of $\PSL(2,\R)$ to the two sheeted cover $\PSL_2(2,\R)$.

\subsection{Hyperbolic elements} \label{subsec:hyperbolic} Let $f,g\in \Diff(\Ss^1)$ be such that:
\begin{itemize} \item $f$ has three fixed points $a_1<b_1<c_1<a_1$
\item $g$ has two fixed points $a_2,b_2$
\item $f'(a_1)g'(b_2)=1$ and $f'(b_1)g'(a_2)=1$
\item $f'(a_1)<1$ and $f'(b_1)>1$
\end{itemize}

\begin{figure}[h] \label{fig:non_convergence_dynamics}
\begin{tikzpicture}[line cap=round,line join=round,>=triangle 45,x=1.0cm,y=1.0cm,scale=0.55]
\clip(-8,-2.5) rectangle (24.48,6);
\draw(1,2) circle (3cm);
\draw(8,2) circle (3cm);
\draw (-0.41,-0.43)-- (-0.33,-0.69);
\draw (-0.59,-0.73)-- (-0.33,-0.69);
\draw (-0.9,4.32)-- (-0.82,4.64);
\draw (-0.65,4.31)-- (-0.9,4.32);
\draw (3.73,2.7)-- (3.96,2.51);
\draw (4.07,2.82)-- (3.96,2.51);
\draw (5.13,3.35)-- (5.42,3.52);
\draw (5.47,3.27)-- (5.42,3.52);
\draw (10.73,0.32)-- (10.7,0.7);
\draw (10.38,0.54)-- (10.7,0.7);
\draw (2.3,5.5) node[anchor=north west] {$b_1$};
\draw (2.9,-0.15) node[anchor=north west] {$a_1$};
\draw (-2.9,2.1) node[anchor=north west] {$c_1$};
\draw (6.8,-0.9) node[anchor=north west] {$b_2$};
\draw (8.7,5.8) node[anchor=north west] {$a_2$};
\draw (0.64,-1.4) node[anchor=north west] {$f$};
\draw (7.97,-1.37) node[anchor=north west] {$g$};
\begin{scriptsize}
\draw [color=black] (2.88,-0.34)-- ++(-1.5pt,-1.5pt) -- ++(3.0pt,3.0pt) ++(-3.0pt,0) -- ++(3.0pt,-3.0pt);
\draw [color=black] (-1.99,1.76)-- ++(-1.5pt,0 pt) -- ++(3.0pt,0 pt) ++(-1.5pt,-1.5pt) -- ++(0 pt,3.0pt);
\draw [color=black] (2.53,4.58)-- ++(-1.5pt,-1.5pt) -- ++(3.0pt,3.0pt) ++(-3.0pt,0) -- ++(3.0pt,-3.0pt);
\draw [color=black] (7.15,-0.88)-- ++(-1.5pt,0 pt) -- ++(3.0pt,0 pt) ++(-1.5pt,-1.5pt) -- ++(0 pt,3.0pt);
\draw [color=black] (8.79,4.89)-- ++(-1.5pt,0 pt) -- ++(3.0pt,0 pt) ++(-1.5pt,-1.5pt) -- ++(0 pt,3.0pt);
\end{scriptsize}
\end{tikzpicture}
\caption{Dynamics of $f$ and $g$}
\end{figure}
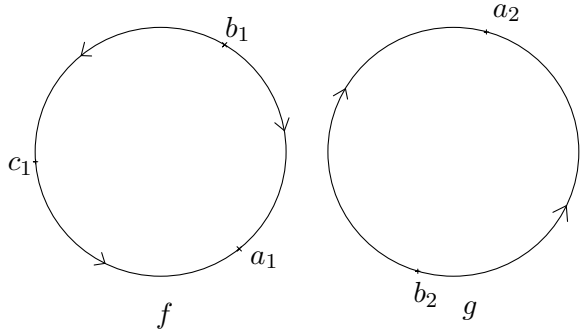

Let $h:\Ss^1\to \Ss^1$ be defined by:
\begin{itemize} \item $h(x)=b_2$ for $x\in \intfo{b_1}{c_1}$
\item $h(x)=a_2$ for $x\in \intfo{c_1}{a_1}$
\item $h : \intff{a_1}{b_1}\to \intff{a_2}{b_2}$ is an orientation preserving homeomorphism (or a non decreasing map such that $h(a_1)=a_2$ and $h(b_1)=b_2$) such that $g\circ h=h\circ f$
\end{itemize}

\begin{prop} \label{elementary_non_convergence} The map $(x,y)\mapsto \p(x,y)=(f(x),g(y))$ preserves a continuous volume form on $M_h$. \end{prop}

\begin{proof}

We start by dividing $M_h$ into several subsets, as shown in Figure \ref{fig:elementary_non_convergence}. Let $\alpha_1, \alpha_2 : \intff{a_1}{b_1} \to \intff{b_2}{a_2}$,  $\beta : \intff{b_1}{c_1} \to \intff{a_2}{b_2}$ and  $\gamma : \intff{c_1}{a_1} \to \intff{a_2}{b_2}$  be decreasing homeomorphisms  whose graphs are invariant under $\p$ (i.e. that conjugate $f$ and $g$). We choose $\alpha_1$ and $\alpha_2$ so that $b_2<\alpha_1(x)<\alpha_2(x)<a_2<b_2$ for all $x\in \intoo{a_1}{b_1}$.\\
\begin{figure}[h] 
\definecolor{aqaqaq}{rgb}{0.63,0.63,0.63}
\definecolor{wqwqwq}{rgb}{0.38,0.38,0.38}
\begin{tikzpicture}[line cap=round,line join=round,>=triangle 45,x=1.0cm,y=1.0cm]
\clip(-4.3,-3.5) rectangle (16.9,5);
\draw [color=wqwqwq] (-2,5)-- (6,5);
\draw [color=wqwqwq] (6,5)-- (6,-3);
\draw [color=wqwqwq] (6,-3)-- (-2,-3);
\draw [color=wqwqwq] (-2,-3)-- (-2,5);
\draw [line width=1.2pt] (3,5)-- (3,0);
\draw [line width=1.2pt] (1,0)-- (3,0);
\draw [line width=1.2pt] (3,5)-- (6,5);
\draw [color=aqaqaq] (1,5)-- (1,-3);
\draw [color=aqaqaq] (3,-3)-- (3,0);
\draw [color=aqaqaq] (-2,0)-- (1,0);
\draw [color=aqaqaq] (3,0)-- (6,0);
\draw [shift={(-3.97,1.97)},line width=1.2pt]  plot[domain=5.09:5.91,variable=\t]({1*5.34*cos(\t r)+0*5.34*sin(\t r)},{0*5.34*cos(\t r)+1*5.34*sin(\t r)});
\draw [shift={(3.74,5.04)},dash pattern=on 3pt off 3pt]  plot[domain=3.15:4.21,variable=\t]({1*5.74*cos(\t r)+0*5.74*sin(\t r)},{0*5.74*cos(\t r)+1*5.74*sin(\t r)});
\draw [shift={(-13.03,-5.02)},dash pattern=on 3pt off 3pt]  plot[domain=0.34:0.74,variable=\t]({1*14.89*cos(\t r)+0*14.89*sin(\t r)},{0*14.89*cos(\t r)+1*14.89*sin(\t r)});
\draw [shift={(-6.96,-7.47)},dash pattern=on 3pt off 3pt]  plot[domain=0.42:0.75,variable=\t]({1*10.92*cos(\t r)+0*10.92*sin(\t r)},{0*10.92*cos(\t r)+1*10.92*sin(\t r)});
\draw [shift={(-1.31,-7.31)},dash pattern=on 3pt off 3pt]  plot[domain=0.53:1.04,variable=\t]({1*8.49*cos(\t r)+0*8.49*sin(\t r)},{0*8.49*cos(\t r)+1*8.49*sin(\t r)});
\draw (-2.2,-3) node[anchor=north west] {$a_1$};
\draw (0.75,-3) node[anchor=north west] {$b_1$};
\draw (2.8,-3.1) node[anchor=north west] {$c_1$};
\draw (-2.6,-2.7) node[anchor=north west] {$a_2$};
\draw (-2.6,0.3) node[anchor=north west] {$b_2$};
\draw (-1.18,1.4) node[anchor=north west] {$\alpha_1$};
\draw (-0.48,3.78) node[anchor=north west] {$\alpha_2$};
\draw (-0.8,2.3) node[anchor=north west] {$W$};
\draw (2.08,-0.66) node[anchor=north west] {$\beta$};
\draw (4.78,-0.64) node[anchor=north west] {$\gamma$};
\draw (3.38,3.36) node[anchor=north west] {$G(h)$};
\draw (5.04,0.98) node[anchor=north west] {$V$};
\draw (0.36,4.96) node[anchor=north west] {$U$};
\draw (3.52,-1.24) node[anchor=north west] {$X$};
\end{tikzpicture}
\caption{Constructing $\omega$ on $M_h$} \label{fig:elementary_non_convergence}
\end{figure}
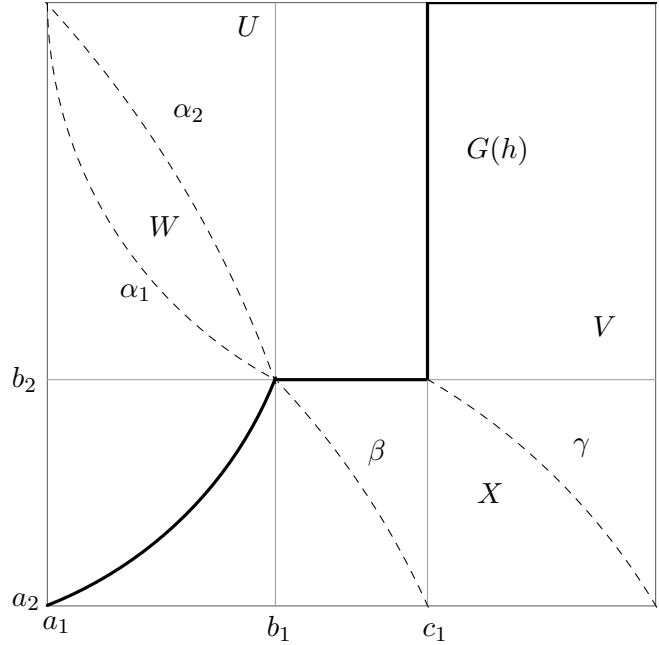
\indent Let $U$, $V$, $W$, $X$ be the open subsets of $M_h$ delimited by $G(h)$ and the graphs of $\alpha_1$, $\alpha_2$, $\beta$ and $\gamma$:
\begin{eqnarray*} U&=& \{(x,y)\in \intof{a_1}{b_1}\times \Ss^1 \vert y\in \intoo{\alpha_2(x)}{h(x)}\} \\& & ~~~~~~~\cup ~\{(x,y)\in \intfo{b_1}{c_1} \times \Ss^1 \vert y\in \intoo{b_2}{\beta(x)} \}\\
V &=& \{(x,y)\in \intof{c_1}{a_1}\times \Ss^1 \vert y\in \intoo{\gamma(x)}{a_2} \} \\& & ~~~~~~~ \cup~ \{ (x,y)\in \intfo{a_1}{b_1}\times \Ss^1 \vert  y\in \intoo{h(x)}{\alpha_1(x)}\}\\
W &=& \{ (x,y)\in \intoo{a_1}{b_1} \times \Ss^1 \vert y\in \intoo{\alpha_1(x)}{\alpha_2(x)}\}\\
X &=& \{(x,y) \in \intof{b_1}{c_1} \times \Ss^1 \vert y\in \intoo{\beta(x)}{b_2}\}\\& & ~~~~~~~ \cup~ \{(x,y)\in \intfo{c_1}{a_1}\times \Ss^1 \vert y\in \intoo{a_2}{\gamma(x)}\} 
\end{eqnarray*}

\indent Consider the linearisation maps $\tau_{a_1}^f : \intoo{c_1}{b_1}\to \R$, $\tau_{b_1}^f : \intoo{a_1}{c_1} \to \R$, $\tau_{a_2}^g : \Ss^1\setminus \{b_2\} \to \R$ and $\tau_{b_2}^g:\Ss^1\setminus \{a_2\}\to \R$. They are smooth maps such that: 
\begin{eqnarray*} \tau_{a_1}^f\circ f\circ (\tau_{a_1}^f)^{-1}(x) &=&  f'(a_1) x \\ \tau_{b_1}^f\circ f\circ (\tau_{b_1}^f)^{-1}(x)&=& f'(b_1) x \\ \tau_{a_2}^g\circ g\circ (\tau_{a_2}^g)^{-1}(x)&=& g'(a_2) x \\ \tau_{b_2}^g\circ g\circ (\tau_{b_2}^g)^{-1}(x)&=& g'(b_2) x \end{eqnarray*}
\indent The map $(x,y)\mapsto (\tau_{b_1}^f(x),\tau_{a_2}^g(y))$ sends $U$ to an open set of $\R^2$ and conjugates $\p$ with $(x,y)\mapsto (f'(b_1)x,g'(a_2)y)$. Since $f'(b_1)g'(a_2)=1$, this map preserves $dx\wedge dy$ on $\R^2$, and $\p$ preserves the pull-back $\omega_U$ on $U$.\\
\indent The map $(x,y)\mapsto (\tau_{a_1}^f(x),\tau_{b_2}^g(y))$ sends $V$ to an open set of $\R^2$ and conjugates $\p$ with $(x,y)\mapsto (f'(a_1)x,g'(b_2)y)$. Since $f'(a_1)g'(b_2)=1$, this map preserves $dx\wedge dy$ on $\R^2$, and $\p$ preserves the pull-back $\omega_V$ on $V$.\\
\indent To extend $\omega$ to $W$, we notice that the action of $\p$ on $\intoo{a_1}{b_1}\times \intoo{b_2}{a_2}$ is conjugate to a translation in $\R^2$, so the quotient is diffeomorphic to a cylinder. The images of $U$ and $V$ on the cylinder are open sets $\hat U$, $\hat V$ each bounded by a curve on which  volume forms $\hat \omega_U$, $\hat \omega_V$ are defined. Simply consider a continuous volume form $\hat \omega$ on the cylinder that is equal to $\hat \omega_U$  on $\hat U$, and equal to $\hat \omega_V$ on $\hat V$ (this is possible because $\hat \omega_U$ and $\hat \omega_V$ can be defined on open sets larger than $U$ and $V$). This lifts to a volume form $\omega$ on $\intoo{a_1}{b_1}\times \intoo{b_2}{a_2}$ that is invariant under $\p$ and that is equal to $\omega_U$ in a neighbourhood of the axes $\{a_1\} \times \intoo{b_2}{a_2}\cup \intoo{a_1}{b_1} \times \{b_2\}$ and to $\omega_V$ on a neighbourhood of $\{b_1\}\times \intoo{b_2}{a_2} \cup \intoo{a_1}{b_1}\times\{a_2\}$.\\
\indent We now only have to extend $\omega$ to $X$. If $x\in \intoo{a_2}{b_2}$, then $\intoo{b_1}{a_1}\times \intfo{x}{g(x)}$ is a fundamental domain for the action of $\p$ on $\intoo{b_1}{a_1}\times \intoo{a_2}{b_2}$, which shows that the quotient is diffeomorphic to the cylinder, and we can extend $\omega$ to $X$ in the same way that we did for $W$.
\end{proof}

The same could also be done for $f$ with four fixed points $a_1<b_1<c_1<d_1<a_1$ such that $b_1$ and $d_1$ are parabolic, $a_1$ is attractive and $c_1$ repulsive. We then choose $g$ with two hyperbolic fixed points $a_2,b_2$ such that $f'(a_1)g'(b_2)=1$ and $f'(c_1)g'(a_1)=1$. We set $h\equiv a_2$ on $\intfo{d_1}{b_1}$ and $h\equiv b_2$ on $\intfo{b_1}{d_1}$. The same kind of division of $M_h$ into four invariant open sets $U,V,W,X$ gives a $(f,g)$-invariant volume form on $M_h$ (see Figure \ref{fig:four_fixed_points}).

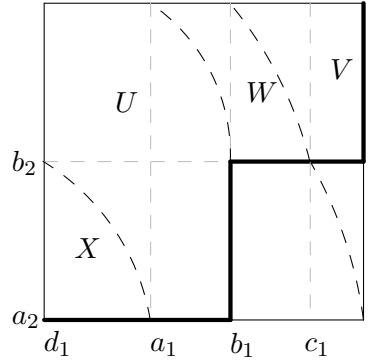
\begin{figure}[h]
\definecolor{cqcqcq}{rgb}{0.75,0.75,0.75}
\begin{tikzpicture}[line cap=round,line join=round,>=triangle 45,x=1.0cm,y=1.0cm,scale=0.7]
\clip(-6.5,-2) rectangle (16.9,6.3);
\draw (-1,5)-- (5,5);
\draw (5,5)-- (5,-1);
\draw (5,-1)-- (-1,-1);
\draw (-1,-1)-- (-1,5);
\draw [line width=1.6pt] (-1,-1)-- (2.5,-1);
\draw [line width=1.6pt] (2.5,-1)-- (2.5,2);
\draw [line width=1.6pt] (2.5,2)-- (5,2);
\draw [line width=1.6pt] (5,2)-- (5,5);
\draw [dash pattern=on 5pt off 5pt,color=cqcqcq] (4,5)-- (4,-1);
\draw [dash pattern=on 5pt off 5pt,color=cqcqcq] (2.5,5)-- (2.5,2);
\draw [dash pattern=on 5pt off 5pt,color=cqcqcq] (1,5)-- (1,-1);
\draw [dash pattern=on 5pt off 5pt,color=cqcqcq] (-1,2)-- (2.5,2);
\draw [shift={(-3.3,-1.7)},dash pattern=on 5pt off 5pt]  plot[domain=0.16:1.02,variable=\t]({1*4.35*cos(\t r)+0*4.35*sin(\t r)},{0*4.35*cos(\t r)+1*4.35*sin(\t r)});
\draw [shift={(-0.99,2.13)},dash pattern=on 5pt off 5pt]  plot[domain=-0.04:0.96,variable=\t]({1*3.49*cos(\t r)+0*3.49*sin(\t r)},{0*3.49*cos(\t r)+1*3.49*sin(\t r)});
\draw [shift={(-3.35,0.2)},dash pattern=on 5pt off 5pt]  plot[domain=0.24:0.69,variable=\t]({1*7.56*cos(\t r)+0*7.56*sin(\t r)},{0*7.56*cos(\t r)+1*7.56*sin(\t r)});
\draw [shift={(-3.24,-2.08)},dash pattern=on 5pt off 5pt]  plot[domain=0.13:0.51,variable=\t]({1*8.31*cos(\t r)+0*8.31*sin(\t r)},{0*8.31*cos(\t r)+1*8.31*sin(\t r)});
\draw (-1.2,-1.02) node[anchor=north west] {$d_1$};
\draw (0.8,-1.16) node[anchor=north west] {$a_1$};
\draw (2.3,-1.05) node[anchor=north west] {$b_1$};
\draw (3.7,-1.16) node[anchor=north west] {$c_1$};
\draw (-1.8,-0.64) node[anchor=north west] {$a_2$};
\draw (-1.8,2.4) node[anchor=north west] {$b_2$};
\draw (0.14,3.5) node[anchor=north west] {$U$};
\draw (4.18,4.14) node[anchor=north west] {$V$};
\draw (2.6,3.7) node[anchor=north west] {$W$};
\draw (-0.62,0.74) node[anchor=north west] {$X$};
\end{tikzpicture}
\caption{Example with four fixed points} \label{fig:four_fixed_points}
\end{figure}

\subsection{Parabolic elements} \label{subsec:parabolic} Let $\gamma \in \PSL(2,\R)$  be parabolic, and let $x_0\in \Ss^1$ be its fixed point. We will denote by $\omega_0$ the volume form on $\Ss^1\times \Ss^1\setminus \Delta$ preserved by $\PSL(2,\R)$.\\
\indent Let $I\subset \Ss^1$ be a compact interval and define a continuous function $h:\Ss^1\to \Ss^1$ such that $h(I)=\{x_0\}$ et $h:\Ss^1\setminus I\to \Ss^1\setminus \{x_0\}$ is an affine diffeomorphism. It is non decreasing of degree one.\\
\indent We can define  $f\in \Diff^{1,1}(\Ss^1)$ such that the restriction to $I$ is an orientation preserving diffeomorphism with $f'=1$ at the endpoints and $h^{-1}\circ \gamma \circ h$ on the complement of $I$. We have $h\circ f=\gamma \circ h$, so the map $(f,\gamma)$ acts on $M_h$.
\begin{prop} The map $(x,y)\mapsto (f(x),\gamma(y))$ preserves a continuous volume form on $M_h$. \end{prop}
\begin{proof} First, define $\omega$ on  $M_h\setminus (\mathring I\times \Ss^1)$ to be $\omega(x,y)=\omega_0(h(x),y)$. It is a continuous volume form (even Lipschitz), and let us show that  $(x,y)\mapsto (f(x),\gamma(y))$ preserves $\omega$. If $x\notin \mathring I$, then we get:

\begin{eqnarray*}  \omega(f(x),\gamma(y))f'(x)\gamma'(y) &=& \omega_0(h\circ f(x),\gamma(y))f'(x)\gamma'(y)\\ &=& \omega_0(\gamma(h(x)),\gamma(y)) \gamma'(h(x))\gamma'(y) \frac{f'(x)}{\gamma'(h(x))}\\ &=& \omega_0(h(x),y) \frac{f'(x)}{\gamma'(h(x))}\\ &=& \omega(x,y) \frac{h'(x)}{h'(f(x))} \\ &=& \omega(x,y)
\end{eqnarray*}
\indent To extend $\omega$ to $ M_h$, we simply notice that the quotient of $I\times \Ss^1 \setminus G(h)$ by $(f,\gamma)$ is diffeomorphic to a cylinder, on which we have defined a volume form on the boundary. We can extend it to a volume form on the cylinder, then lift it to a continuous volume form on $I\times \Ss^1\setminus G(h)$ that is $(f,\gamma)$-invariant and equal to $\omega$ on the boundary. This way, we have defined a continuous volume form on $M_h$ that is $(f,\gamma)$-invariant. \end{proof}

\begin{rem} In all of the other examples, the invariant volume forms can actually be constructed in a smooth way. However, in the case of the opening of the fixed point of a parabolic, it is not clear whether it is possible to find a smooth volume form.
\end{rem}

\subsection{Lifts of parabolics to $\PSL_2(2,\R)$} The two-sheeted covering of the circle is still a circle, which induces a two-sheeted covering group $\PSL_2(2,\R)$ of $\PSL(2,\R)$. Lifts of parabolic elements either have two fixed points, or two (not fixed) periodic points. Let $R\in \PSL_2(2,\R)$ be a generator of the centre ($R$ is of order two).\\
\indent If $\gamma\in \PSL_2(2,\R)$ is a lift of a parabolic element without fixed points, then it has two periodic points $x_0,R(x_0)$. Let $U=\{ (x,y) \vert y\in \intoo{x}{R(x)}\}$. The quotient by the map $(x,y)\mapsto (R(x),R(y))$ of $U$ endowed with the action of $\PSL(2,\R)$ is equivariant to the diagonal action on $\Ss^1\times \Ss^1\setminus \Delta$, therefore has a volume form preserved by $\PSL(2,\R)$, which can be lifted to $U$ as a volume form $\omega_0$ invariant under $\PSL_2(2,\R)$.\\
\indent Let $h:\Ss^1\to \Ss^1$ be defined by $h(x)=x_0$ if $x\in \intfo{x_0}{R(x_0)}$ and $h(x)=R(x_0)$ if $x\in \intfo{R(x_0)}{x_0}$. It is non decreasing of degree one, and it commutes with $\gamma$.

\begin{prop} The map $(x,y)\mapsto (\gamma(x),\gamma(y))$ preserves a continuous volume form on $M_h$. \end{prop}
\begin{proof} In order to extend $\omega_0$ to $M_h$, we start by considering the open set $V=\{(x,y)\vert x\in \intoo{x_0}{R(x_0)} , y\in \intoo{x}{x_0}\}$ and its image $R(V)$ under the map $(x,y)\mapsto (R(x),R(y))$ (see Figure \ref{fig:lift}).
\begin{figure}[h]
\definecolor{cqcqcq}{rgb}{0.75,0.75,0.75}
\begin{tikzpicture}[line cap=round,line join=round,>=triangle 45,x=1.0cm,y=1.0cm,scale=0.8]
\clip(-2.1,1) rectangle (16.9,6.3);
\fill[fill=black,pattern=vertical lines] (-2,1) -- (0,1) -- (2,3) -- (2,5) -- cycle;
\fill[fill=black,pattern=vertical lines] (-2,3) -- (0,5) -- (-2,5) -- cycle;
\fill[fill=black,pattern=vertical lines] (4,1) -- (6,1) -- (6,3) -- cycle;
\fill[fill=black,pattern=vertical lines] (6,3) -- (8,3) -- (8,5) -- cycle;
\draw (-2,5)-- (-2,1);
\draw (2,1)-- (-2,1);
\draw (2,5)-- (2,1);
\draw (-2,5)-- (2,5);
\draw (4,5)-- (8,5);
\draw (4,5)-- (4,1);
\draw (4,1)-- (8,1);
\draw (8,5)-- (8,1);
\draw (10,5)-- (10,1);
\draw (10,1)-- (14,1);
\draw (14,5)-- (14,1);
\draw (10,5)-- (14,5);
\draw [line width=1.6pt] (-2,1)-- (2,5);
\draw (0,5)-- (-2,3);
\draw [line width=1.6pt] (2,3)-- (0,1);
\draw (4,1)-- (8,5);
\draw [dash pattern=on 5pt off 5pt] (8,3)-- (6,1);
\draw [dash pattern=on 5pt off 5pt] (6,5)-- (4,3);
\draw (-2,1)-- (0,1);
\draw (0,1)-- (2,3);
\draw (2,3)-- (2,5);
\draw [line width=1.6pt] (2,5)-- (-2,1);
\draw [line width=1.6pt] (-2,3)-- (0,5);
\draw (0,5)-- (-2,5);
\draw (-2,5)-- (-2,3);
\draw (6,3)-- (6,1);
\draw [line width=1.6pt] (6,3)-- (8,3);
\draw [line width=1.6pt] (4,5)-- (4,3);
\draw [line width=1.6pt] (4,5)-- (6,5);
\draw (4,1)-- (6,1);
\draw [line width=1.6pt] (6,1)-- (6,3);
\draw (6,3)-- (4,1);
\draw (6,3)-- (8,3);
\draw (8,3)-- (8,5);
\draw (8,5)-- (6,3);
\draw [line width=1.6pt] (-2,5)-- (0,5);
\draw [line width=1.6pt] (-2,5)-- (-2,3);
\draw [line width=1.6pt] (0,1)-- (-2,1);
\draw [line width=1.6pt] (2,5)-- (2,3);
\draw [line width=1.6pt] (4,1)-- (6,1);
\draw [line width=1.6pt] (4,1)-- (6,3);
\draw [line width=1.6pt] (6,3)-- (8,5);
\draw [line width=1.6pt] (8,5)-- (8,3);
\draw [line width=1.6pt] (10,5)-- (10,3);
\draw [line width=1.6pt] (10,5)-- (12,5);
\draw [line width=1.6pt] (12,3)-- (12,1);
\draw [line width=1.6pt] (12,3)-- (14,3);
\draw [line width=1.6pt] (10,1)-- (12,1);
\draw [line width=1.6pt] (14,5)-- (14,3);
\draw [dash pattern=on 5pt off 5pt] (10,1)-- (14,5);
\draw [dash pattern=on 5pt off 5pt] (10,3)-- (12,5);
\draw [dash pattern=on 5pt off 5pt] (14,3)-- (12,1);
\draw (-0.74,3.88) node[anchor=north west] {$U$};
\draw (4.38,4.62) node[anchor=north west] {$V$};
\draw (6.2,2.78) node[anchor=north west] {$R(V)$};
\draw [dash pattern=on 5pt off 5pt,color=cqcqcq] (6,5)-- (6,3);
\draw (10.9,1.98) node[anchor=north west] {$W$};
\draw (12.5,4.04) node[anchor=north west] {$R(W)$};
\end{tikzpicture}
\caption{Invariant volume form for a parabolic element of $\PSL_2(2,\R)$} \label{fig:lift}
\end{figure}
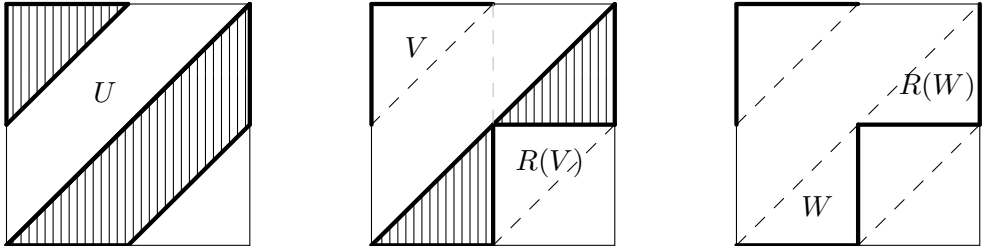

Since the action of $\gamma$ on $V\cup R(V)$ is proper (the quotient is a cylinder), we can extend $\omega_0$ from $U\cap (V\cup R(V))$ to a volume form $\omega$ on $U\cup V\cup R(V)$.\\
\indent Similarly, we set $W=\{(x,y)\vert x\in \intoo{x_0}{R(x_0)} , y\ne x_0\}$ and extend $\omega$ to $W\cup R(W)$, so that $\omega$ is now defined on $M_h$.
\end{proof}

\subsection{Classification of elements of $\rho_1(G_\omega)$ up to topological conjugacy}
We will now see that we have described all of the examples. 
\begin{prop} \label{classification_elements} Let $h:\Ss^1\to \Ss^1$ be non decreasing of degree one, and let $\omega$ be a continuous volume form on $M_h$. For any $\p\in G_\omega$, $\rho_1(\p)$ is topologically conjugate to one of the examples described above, i.e. it satisfies one of the following propositions:
\begin{itemize} \item $\rho_1(\p)$ is topologically conjugate to an element of $\PSL(2,\R)$.
\item $\rho_1(\p)$ is topologically conjugate to a parabolic element of $\PSL_2(2,\R)$.
\item $\rho_1(\p)$ has three fixed points $a<b<c<a$ such that $a,b$ are hyperbolic and $c$ is parabolic.
\item $\rho_1(\p)$ has four fixed points $a<b<c<d<a$ such that $a,c$ are hyperbolic and $b,d$ are parabolic. 
\end{itemize}
\end{prop}
Let $h:\Ss^1\to \Ss^1$ be non decreasing of degree one,  let $\omega$ be a continuous volume form on $M_h$, and let $\p=(f,g)\in G_\omega\setminus \{Id\}$. We will classify them according to their numbers of fixed points.\\
\indent As we saw in Lemma \ref{continuous_rigidity}, fixed points of $\p$ in $M_h$ are hyperbolic. This implies that there cannot be too many of them.
\begin{lemma} \label{number_fixed_points} If $a_1\in \Ss^1$ is fixed by $f$, then there is at most one fixed point $b_2\in \Ss^1$ of $g$ such that $(a_1,b_2)\in M_h$. \end{lemma}
\begin{proof} Assume that there are two fixed points $a_2,b_2$ of $g$ such that $(a_1,a_2)\in M_h$ and $(a_1,b_2)\in M_h$. According to Lemma \ref{continuous_rigidity}, $a_1$ is hyperbolic, so up to replacing $\p$ by $\p^{-1}$ we can assume that $a_1$ is attractive for $f$. This implies that $a_2$ and $b_2$ are both repulsive for $g$, so $g$ has non repulsive fixed points $c_2\in \intoo{a_2}{b_2}$ and $d_2\in \intoo{b_2}{a_2}$. Since they are not repulsive, $(a_1,c_2)\notin M_h$ and $(a_1,d_2)\notin M_h$. This is absurd because the set of $y\in \Ss^1$ such that $(a_1,y)\in M_h$ is connected. \end{proof}

\begin{proof}[Proof of Proposition \ref{classification_elements}] We will distinguish different cases depending on the number of fixed points of $f$ and $g$.\\

\noindent \textbf{One fixed point:} First, we assume that  $g$ has exactly one fixed point $y_0\in \Ss^1$. The interval $h^{-1}(\{y_0\})$ is stabilised by $f$, so its endpoints are fixed by $f$. If $f$ had a fixed point $x$ outside of $h^{-1}(\{y_0\})$, then $h(x)$ is fixed by $g$, which is impossible, therefore $f$ has no fixed point out of $h^{-1}(\{y_0\})$, and it is topologically conjugate to an example described in \ref{subsec:parabolic}.\\

\noindent \textbf{Two fixed points:} Now, assume that both $f$ and $g$ have at least two fixed points. Let $a_1\ne b_1$ be fixed points of $f$ and $a_2\ne b_2$ be fixed points of $g$.\\
\indent  First, let us assume that $\p$ has no fixed point in $M_h$. In this case,  non of the four points $(a_1,a_2)$, $(a_1,b_2)$, $(b_1,a_2)$, $(b_1,b_2)$ are in $M_h$. This implies that $h$ is constant on both intervals $\intoo{a_1}{b_1}$ and $\intoo{b_1}{a_1}$, therefore $f$ (as well as $g$) has exactly two fixed points, so it either has north/south dynamics (in which case it is topologically conjugate to a hyperbolic element of $\PSL(2,\R)$), or it is topologically conjugate to a parabolic element of $\PSL_2(2,\R)$.\\
\indent We can now assume that $\p$ has a fixed point in $M_h$. Up to renaming the fixed points, we can assume that $(a_1,b_2)\in M_h$. Lemma \ref{continuous_rigidity} implies that they are hyperbolic fixed points for $f$ and $g$. If they each have only two fixed points, then they have north/south dynamics, and are topologically conjugate to a hyperbolic element in $\PSL(2,\R)$. Up to replacing $\p$ by $\p^{-1}$, we can assume that $a_1$ is attractive for $f$ and $b_2$ is repulsive for $g$. \\
\indent Assume that $f$ has a third fixed point $c_1\in \intoo{b_1}{a_1}$. Because of Lemma \ref{number_fixed_points}, the points $(b_1,b_2)$ and $(c_1,b_2)$ are not in $M_h$, which implies that $h\equiv b_2$ on $\intoo{b_1}{c_1}$. Up to replacing $b_1$ and $c_1$ by the edges of the interval $h^{-1}(\{b_2\})$ (which are fixed by $f$), we can assume that $\overline{h^{-1}(\{b_2\})}=\intff{b_1}{c_1}$. Lemma \ref{number_fixed_points} implies that $a_1$ is the only fixed point of $f$ in $\intoo{c_1}{a_1}$, and that $f$ has at most one fixed point in $\intoo{b_1}{c_1}$, and that if it exists, then it is hyperbolic. This implies that $f$ has either three or four fixed points and is topologically conjugate to one of the examples described in \ref{subsec:hyperbolic}.\\
\indent Notice that in this case (where $f$ has at least three fixed points), $g$ has only two fixed points (an extra fixed point would generate a point in $M_h$ which cannot exist because of Lemma \ref{number_fixed_points}).\\

\noindent \textbf{No fixed points:} We have treated all cases where $f$ and $g$ have fixed points (note that since $f$ and $g$ are semi conjugate, if one has a fixed point then so does the other). They share the same rotation number $\alpha \in \R/\Z$. If they have no fixed points, then $\alpha \ne 0$.\\
\indent If $\alpha=\frac{p}{q}$ is rational, then the number of fixed points of  $f^q$ and $g^q$ are at least $q$ and are multiples of $q$, and $(f^q,g^q)\in G_\omega$. This shows that if $f^q\ne Id$, then $\alpha=\frac{1}{2}$ (because either $f^q$ or $g^q$ has exactly two fixed points) and $f^2$  either has north/south dynamics, has four fixed points, or is topologically conjugate to a parabolic element of $\PSL_2(2,\R)$. In the latter case, $f$ is itself topologically conjugate to a parabolic element of $\PSL_2(2,\R)$.\\ \indent If $f^2$ has north/south dynamics, let $a_1,b_1$ be its fixed points. Then $f$ conjugates $f^2$ on a neighbourhood of $a_1$ with $f^2$ on a neighbourhood of $a_2$, which is absurd because one is attractive and the other is repulsive.\\
\indent  If $f^2$ has some hyperbolic points (which is the case if it has four fixed points), then let $a_1$ be a   hyperbolic periodic points of $f$ and $b_1=f(a_1)$. They satisfy $(f^2)'(a_1)=f'(b_1)f'(a_1)=(f^2)'(b_1)$. This is impossible since one is attractive for $f^2$ and the other is repulsive. Therefore, if $f$ has no fixed points and $\alpha=\frac{p}{q} \in \Q$, then either $f^q=Id$, and $f$ is topologically conjugate to a rotation, either $f$ is topologically conjugate to a parabolic element of $\PSL_2(2,\R)$.\\
\indent If $\alpha \notin \Q$, then there are two possibilities: either $f$ is topologically conjugate to a rotation, either $f$ has an invariant Cantor set and is semi conjugate to a rotation (a Denjoy example, which cannot be the case if $f$ is $C^2$, see \cite{navas_book} for a thorough treatment of Denjoy diffeomorphisms). We are going to show that $f$ and $g$ are both topologically conjugate to a rotation.\\
\indent Assume that $f$ is a Denjoy example. Let $K\subset \Ss^1$ be the invariant Cantor set, and let $I$ be a connected component of $\Ss^1\setminus K$. Using the fact that $f^n(I)\cap f^p(I)=\emptyset$ if $n\ne p$, a simple calculation (see exercise 3.5.24 in \cite{navas_book}) shows that:
$$ \int_I \left(\sum_{n\in \Z} (f^n)'(x)\right) dx = \sum_{n\in \Z} \vert f^n(I)\vert  \le 1$$
\indent In particular, the set of points  $x\in \Ss^1$  such that $(f^n)'(x)\to 0$ as $n\to \infty$ is a set of full Lebesgue measure in $\Ss^1\setminus K$, hence non empty. The points $y\in \Ss^1$ such that the sequence $((g^n)'(y))_{n\in \Z}$ is bounded form a non empty set (\cite{Herman}, chapter X) invariant under $g$.\\
\indent Note that $G(g\circ h)\subset M_h$. Indeed, if $h$ is continuous at $x$, then $g(h(x))\ne h(x)$, and if $h$ is discontinuous at $x$, then the interval $\intoo{h_l(x)}{h_r(x)}$ is wandering for $g$, therefore cannot intersect its image $g(\intoo{h_l(x)}{h_r(x)})$. Since $\p^2\in G_\omega$, we also have $G(g^2\circ h)\subset M_h$. Up to replacing $\p$ by $\p^{-1}$, we can assume that $K=\{ (x,y)\in \Ss^1\times \Ss^1 \vert y\in \intff{g(h_l(x))}{g^2(h_r(x))}\}$ is a compact subset of $M_h$, invariant under $\p$, with non empty interior. Let $x\in \Ss^1$ be such that $(f^n)'(x)\to 0$ as $n\to \infty$, and let $y\in \Ss^1$ be  such that  the sequence $((g^n)'(y))_{n\in \Z}$ is bounded and $(x,y)\in K$. Let $M=\max_K \omega$ and $m=\min_K \omega$. We then have:
$$\underbrace{(f^n)'(x)(g^n)'(y)}_{\to 0} = \frac{\omega(x,y)}{\omega(f^n(x),g^n(y))} \in \left[ \frac{m}{M},\frac{M}{m} \right] $$ \indent This is a contradiction, therefore $f$ cannot be a Denjoy example, so it is topologically conjugate to a rotation, which is an element of $\PSL(2,\R)$.
\end{proof}

\section{Non elementary examples} \label{sec:examples} This section is dedicated to the proof of Theorem \ref{free_group_example}. The main idea in constructing a non elementary example consists in starting with an appropriately chosen representation $\rho_1:\mathbb F_3\to \PSL(2,\R)$, then considering $\rho_2:\mathbb F_3\to \Diff(\Ss^1)$ where we only modified one of the generators by introducing a parabolic fixed point (just as in the elementary case).   Proposition \ref{elementary_non_convergence} shows that the images of the generators by $(\rho_1,\rho_2)$ each preserve a volume form, and the difficulty consists in showing that we can find one that is preserved by all three. The proof is almost identical to the proof of Theorem 1.4 in \cite{Monclaira}.\\
\indent One of the main ingredients in this proof is the construction of a flow on a $3$-manifold (a deformation of the geodesic flow on $\rm{T}^1\h^2/\rho_1(\Gamma)$) that has a transverse structure given by $(\rho_1,\rho_2)$. This construction follows an idea of Ghys used in two different settings. The first one, found in \cite{Gh93}, was to show a rigidity theorem for actions of surface groups on the circle, and the second was the construction of (the only) exotic Anosov flows with smooth weak stabe and weak unstable foliations on 3-manifolds  in \cite{Gh92}, called quasi-Fuchsian flows. However, Ghys used a local construction (given a certain atlas on  $\rm{T}^1\h^2/\rho_1(\Gamma)$), whereas we will take a global approach.

\subsection{Hyperbolic flows} \label{subsec:hyperbolic_flows}  Let us recall a few basic notions of hyperbolic flows. Let $\p^t$ be a complete flow generated by a vector field $X$ on a manifold $M$. We say that a compact invariant set $K\subset M$ is \textbf{hyperbolic} if there are positive constants $C,\lambda$ and a decomposition of tangent spaces $T_xM=E^s_x\oplus E^u_x\oplus \R.X$ for each $x\in K$ such that:
$$ \forall x\in K~\forall v\in E^s_x~\forall t\ge 0~~\Vert D\p^t_x(v)\Vert \le C e^{-\lambda t} \Vert v\Vert $$
$$\forall x\in K~\forall v\in E^u_x~\forall t\le 0~~\Vert D\p^t_x(v)\Vert \le C e^{\lambda t} \Vert v\Vert
$$
\indent The norm $\Vert . \Vert$ denotes the norm given by any Riemannian metric on $M$ (since $K$ is compact, the definition does not depend on the choice of a Riemannian metric). If the whole manifold $M$ is a hyperbolic set, then we say that $\p^t$ is an Anosov flow.\\
\indent Let $\p^t$ be a smooth flow on a manifold $M$. If $K\subset M$ is a compact hyperbolic set and $x\in K$, then we define the \textbf{stable and unstable manifolds} through $x$:
$$ W^s(x)=\{ z\in M \vert d(\p^t(x),\p^t(z))\underset{t\to +\infty}{\longrightarrow} 0\}$$
$$W^u(x)=\{ z\in M \vert d(\p^t(x),\p^t(z))\underset{t\to -\infty}{\longrightarrow}  0\}$$
\indent The Stable Manifold Theorem states that they are submanifolds of $M$ tangent to $E^s$ and $E^u$ at $x$ (see \cite{HP}).\\
\indent The most important fact for us is that the limit $d(\p^t(x),\p^t(z))\to 0$ is a uniformly decreasing exponential: for all compact set $A$ and all $\e>0$, there is a constant $C'>0$ such that:
$$\forall x\in K ~\forall z\in W^s(x)\cap A~\forall t\ge 0~~ d(\p^t(x),\p^t(z))\le C' e^{-(\lambda-\e) t} $$
$$\forall x\in K ~\forall z\in W^u(x)\cap A~\forall t\le 0~~ d(\p^t(x),\p^t(z))\le C' e^{(\lambda-\e) t} $$
\indent We will denote by $W^s(K)$ (resp. $W^u(K)$) the union $W^s(K)=\bigcup_{x\in K} W^s(x)$ (resp. $W^u(K)=\bigcup_{x\in K} W^u(x)$).

\subsection{A cohomological reformulation} Searching for an invariant volume form is equivalent to solving a cohomological equation. Let $\omega_0$ be a volume form on a smooth manifold $M$, and let $\rho : \Gamma \to \Diff(M)$ be a group representation. Any other volume form on $M$ is a multiple of $\omega_0$, hence if $\gamma \in \Gamma$, then we can write $\rho(\gamma)^*\omega_0=e^{-\alpha_{\gamma}}\omega_0$.  The chain rule shows that $\alpha_{\gamma}$ satisfies the cocycle relation $\alpha_{\gamma'\gamma}=\alpha_{\gamma'}\circ \rho(\gamma) + \alpha_{\gamma}$.\\
\indent Let $\omega=e^{\sigma}\omega_0$ be a volume form on $M$. We can compute the pull back $\rho(\gamma)^*\omega=e^{\sigma\circ\rho(\gamma)}\rho(\gamma)^*\omega_0=e^{\sigma\circ\rho(\gamma)-\sigma-\alpha_{\gamma}}\omega$, hence $\omega$ is preserved by $\Gamma$ if and only if $\sigma\circ \rho(\gamma)-\sigma = \alpha_{\gamma}$ for all $\gamma\in \Gamma$. In other words, we wish to show that the cocycle $\alpha_{\gamma}$ is a coboundary. \\
\indent The issue with this formulation of the problem is that cohomological equations for the action $(\rho_1,\rho_2)$ on $M_h$ are difficult to solve. We will now see how we can translate the problem to a cohomology equation for a hyperbolic flow, which is a much more simple situation. In this setting, a cocycle is a smooth  function $\alpha:M\to \R$ (where $M$ is the manifold on which we study a flow $\p^t$), and we look for a smooth function $\sigma :M\to \R$ such that $\sigma(\p^t(x))-\sigma(x)=\int_0^t\alpha(\p^s(x))ds$ for all $(x,t)\in M\times \R$.\\
\indent There is a first necessary condition for the existence of a solution: if $x\in \mathrm{Per}(\p)$, i.e. if there is $T>0$ such that $\p^T(x)=x$, then $\int_0^T\alpha(\p^s(x))ds=0$.  Livšic's Theorem states that this condition is sufficient in order to find a solution on a compact hyperbolic set.

\begin{theo} Let $\p^t$ be a smooth flow on a manifold $M$, and let $K$ be a compact hyperbolic set, such that the action on $K$ has a dense orbit. If $\alpha : K\to \R$ is a H\"older continuous function such that $\int_0^T\alpha(\p^s(x))ds=0$ for all $x\in K$ such that $\p^T(x)=x$, then there is a unique H\"older continuous function $\sigma :K\to \R$ such that $\sigma(\p^t(x))-\sigma(x)=\int_0^t\alpha(\p^s(x))ds$ for all $(x,s)\in K\times \R$. \end{theo}

As stated, the proof can be found in \cite{KH} (Livšic's work in \cite{livsic} deals with Anosov flows on compact manifolds).\\
\indent However,  Livšic's Theorem will not be of any use in the proof of Theorem \ref{free_group_example}, because we will already have a solution on the hyperbolic set. Instead, we will show that given a solution on a compact hyperbolic set $K$, we can extend it to $W^s(K)\cup W^u(K)$. When translating the problem back to the action on $M_h$, this will give a volume form invariant at points of $L_{\rho_1(\Gamma)}\times \Ss^1\cup\Ss^1\times L_{\rho_2(\Gamma)}$, and there will still be some work involved in order to extend the solution to $M_h$ (which is the content of subsection \ref{subsec:extending}).

\subsection{Convex cocompact groups and geodesic flows} \label{subsection:convex} Let $\Gamma\subset \PSL(2,\R)$ be a discrete non elementary subgroup such that the limit set $L_\Gamma$ is a Cantor set. The \textbf{convex hull} of $\Gamma$ is the subset $C_\Gamma$ of $\h^2$ bounded by geodesics joining fixed points of hyperbolic elements of $\Gamma$. We say that $\Gamma$ is \textbf{convex cocompact} if $C_\Gamma /\Gamma$ is compact. A particular case of Ahlfors' Finiteness Theorem (see \cite{ahlfors} or \cite{bers}) states that any finitely generated discrete subgroup of $\PSL(2,\R)$ with only hyperbolic elements is convex cocompact.\\
\indent If $\Gamma\subset \PSL(2,\R)$ is convex cocompact, then denote by $\p^t$ the geodesic flow on $\mathrm{T}^1\h^2/\Gamma$ (remark that even if $\h^2/\Gamma$ is not a manifold, the unit bundle $\mathrm{T}^1\h^2/\Gamma$ always is when $\Gamma$ is discrete).\\
\indent The \textbf{non wandering set} $\Omega_\p$ of a flow is the set of points $x$ such that there are sequences $x_n\to x$ and  $t_n\to \infty$ satisfying $\p^{t_n}(x_n)\to x$. For the geodesic flow,  $\Omega_\p$ can be described as follows: its lift to $\rm{T}^1\h^2$ is the set of vectors tangent to a geodesic that lies entirely in $C_\Gamma$. The important property of $\p^t$ is that it is an Axiom A flow: $\Omega_\p$ is a compact hyperbolic set for $\p^t$, and it is equal to the closure of periodic orbits $\mathrm{Per}(\p)$ (Axiom A flows are a generalization of Anosov flows that can be defined even on non compact manifolds). We will now use a presentation of the geodesic flow that is particularly convenient when we define perturbations.

\indent Let $\Sigma_3 = \{(x_-,x_0,x_+)\in (\Ss^1)^3 \vert x_-<x_0<x_+<x_-\}$ be the set of ordered triples of $\Ss^1$. We can identify $\rm{T}^1\h^2$ and $\Sigma_3$ in the following way: given a unit vector $v\in \rm{T}^1\h^2$, we consider $x_-$ and $x_+$ the limits at $-\infty$ and $+\infty$ of the geodesic given by $v$, and $x_0$ is the limit at $+\infty$ of the geodesic passing through the base point of $v$ in an orthogonal direction, oriented to the right of $v$ (see Figure \ref{fig:geodesics}).\\
\begin{figure}[h]
\begin{tikzpicture}[line cap=round,line join=round,>=triangle 45,x=1cm,y=1cm,scale=0.55]
\clip(-9,-4.98) rectangle (16.96,6.3);
\draw(2.52,0.82) circle (5.32cm);
\draw [shift={(10.74,3.34)}] plot[domain=2.77:4.11,variable=\t]({1*6.75*cos(\t r)+0*6.75*sin(\t r)},{0*6.75*cos(\t r)+1*6.75*sin(\t r)});
\draw [shift={(3.3,9.63)}] plot[domain=4.81:5.27,variable=\t]({1*7.01*cos(\t r)+0*7.01*sin(\t r)},{0*7.01*cos(\t r)+1*7.01*sin(\t r)});
\draw [->] (4.02,2.65) -- (3.81,4.72);
\draw (3.1,4) node[anchor=north west] {$v$};
\draw (4.6,6.4) node[anchor=north west] {$x_+$};
\draw (7,-2) node[anchor=north west] {$x_-$};
\draw (7.1,4.1) node[anchor=north west] {$x_0$};
\draw (3.99,3.06)-- (4.3,3.1);
\draw (4.3,3.1)-- (4.32,2.69);
\end{tikzpicture}
\caption{Identification between $\rm{T}^1\h^2$ and $\Sigma_3$} \label{fig:geodesics}
\end{figure}
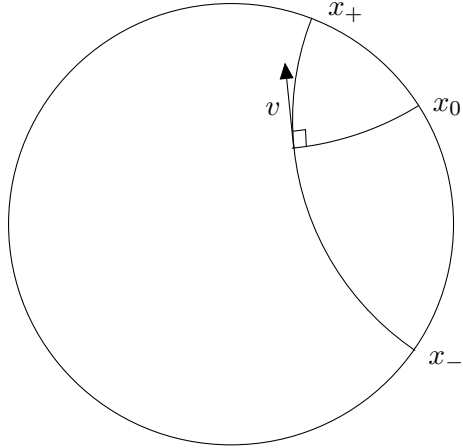

\indent On $\Sigma_3$, the geodesic vector field is a rescaling of the constant vector field $(0,1,0)$, and the action $\alpha$ of $\PSL(2,\R)$ is the diagonal action. The geodesic flow $\p^t$ is defined on the quotient manifold $M=\Sigma_3/\alpha(\Gamma)\approx \rm{T}^1\h^2/\Gamma$. The image of a point $(x_-,x_0,x_+)$ in $M$ is in $\Omega_\p$ if and only if $(x_-,x_+)\in L_{\Gamma}\times L_{\Gamma}$, and it is in $\rm{Per}(\p)$ if and only if  $(x_-,x_+)$ is the pair of fixed points of an element $\gamma\in \Gamma$.

\subsection{Choice of $\rho_1$ and construction of $\rho_2$}  Let $T$ be the twice punctured torus. Its fundamental group is the free group on three generators $\mathbb F_3=\langle a,b,c\rangle$. Given a complete hyperbolic structure on $T$ such that neighbourhoods of the omitted points have infinite volume, we obtain a convex cocompact representation  $\rho_1:\mathbb F_3\to \PSL(2,\R)$.\\
\indent Let $\delta_1$ and $\delta_2$ be the simple loops going around the omitted points (see Figure \ref{fig:punctured_torus}). If $I$ is a connected component of $\Ss^1\setminus L_{\rho_1(\mathbb F_3)}$ and $\gamma$ stabilises $I$, then $\gamma$ is conjugate to a power of $\delta_1$ or $\delta_2$. It is explained in \cite{Button} how the generators can be chosen in a way the $\delta_1=abc$ and  $\delta_2=cba$. This shows that $a,b,\delta_1$ freely generate $\pi_1(T)$.

\begin{figure}[h]
\begin{tikzpicture}[line cap=round,line join=round,>=triangle 45,x=1.0cm,y=1.0cm,scale=0.7]
\clip(-4,-4.3) rectangle (18.44,3);
\draw [shift={(5.02,-6.64)}] plot[domain=1.21:1.93,variable=\t]({1*9.16*cos(\t r)+0*9.16*sin(\t r)},{0*9.16*cos(\t r)+1*9.16*sin(\t r)});
\draw [shift={(12.08,11.77)}] plot[domain=4.34:4.75,variable=\t]({1*10.56*cos(\t r)+0*10.56*sin(\t r)},{0*10.56*cos(\t r)+1*10.56*sin(\t r)});
\draw [shift={(-1.98,11.79)}] plot[domain=4.67:5.08,variable=\t]({1*10.56*cos(\t r)+0*10.56*sin(\t r)},{0*10.56*cos(\t r)+1*10.56*sin(\t r)});
\draw [shift={(5.02,5.12)}] plot[domain=4.35:5.08,variable=\t]({1*9.16*cos(\t r)+0*9.16*sin(\t r)},{0*9.16*cos(\t r)+1*9.16*sin(\t r)});
\draw [shift={(-1.98,-13.31)}] plot[domain=1.2:1.61,variable=\t]({1*10.56*cos(\t r)+0*10.56*sin(\t r)},{0*10.56*cos(\t r)+1*10.56*sin(\t r)});
\draw [shift={(12.08,-13.29)}] plot[domain=1.53:1.94,variable=\t]({1*10.56*cos(\t r)+0*10.56*sin(\t r)},{0*10.56*cos(\t r)+1*10.56*sin(\t r)});
\draw [shift={(4.85,0.83)}] plot[domain=3.69:5.76,variable=\t]({1*2.69*cos(\t r)+0*2.69*sin(\t r)},{0*2.69*cos(\t r)+1*2.69*sin(\t r)});
\draw [shift={(4.76,-3.32)}] plot[domain=0.89:2.2,variable=\t]({1*2.39*cos(\t r)+0*2.39*sin(\t r)},{0*2.39*cos(\t r)+1*2.39*sin(\t r)});
\draw [shift={(-3.59,-0.87)}] plot[domain=-0.56:0.62,variable=\t]({1*3.79*cos(\t r)+0*3.79*sin(\t r)},{0*3.79*cos(\t r)+1*3.79*sin(\t r)});
\draw [shift={(7.44,-0.76)}] plot[domain=-0.59:0.59,variable=\t]({1*3.77*cos(\t r)+0*3.77*sin(\t r)},{0*3.77*cos(\t r)+1*3.77*sin(\t r)});
\draw [shift={(4.23,-0.62)},dash pattern=on 6pt off 6pt]  plot[domain=2.75:3.6,variable=\t]({1*5.12*cos(\t r)+0*5.12*sin(\t r)},{0*5.12*cos(\t r)+1*5.12*sin(\t r)});
\draw [shift={(13.86,-0.76)},dash pattern=on 6pt off 6pt]  plot[domain=2.58:3.71,variable=\t]({1*3.89*cos(\t r)+0*3.89*sin(\t r)},{0*3.89*cos(\t r)+1*3.89*sin(\t r)});
\draw (0.11,0.99) node[anchor=north west] {$\delta_1$};
\draw (11.31,0.99) node[anchor=north west] {$\delta_2$};
\end{tikzpicture}
 \caption{The twice punctured torus} \label{fig:punctured_torus}
\end{figure}
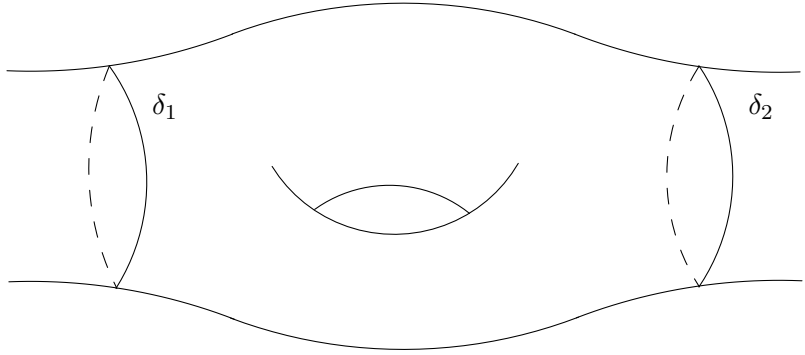

\indent  Let $N,S\in \Ss^1$ be the fixed points of $\rho_1(\delta_1)$, such that $\intoo{N}{S}$ is a connected component of $\Ss^1\setminus L_{\rho_1(\mathbb F_3)}$. Let $f\in \Diff(\Ss^1)$ be equal to $\rho_1(a)$ on $\intff{S}{N}$, and have exactly one fixed point $P_a$ in $\intoo{N}{S}$. We define $\rho_2:\mathbb F_3 \to \Diff(\Ss^1)$ by $\rho_2(a)=\rho_1(a)$, $\rho_2(b)=\rho_1(b)$ and $\rho_2(\delta_1)=f$. \\
\indent Since $f$ coincides with $\rho_1(\delta_1)$ on the limit set, we see that $\rho_2(\gamma)$ and $\rho_1(\gamma)$ coincide on $L_{\rho_1(\mathbb F_3)}$ for all $\gamma \in \mathbb F_3$. More precisely, they are equal everywhere except on intervals that are bounded by images of $N$ and $S$. Therefore, for any $\gamma\in \mathbb F_3$ that is not a power of a conjugate of $\delta_1$, $\rho_2(\gamma)$ has two fixed points, which are hyperbolic.

\begin{lemma} $L_{\rho_2(\mathbb F_3)} = L_{\rho_1(\mathbb F_3)}$. \end{lemma}
\begin{proof}  The compact set   $L_{\rho_1(\mathbb F_3)}$ is invariant under $\rho_2$, which shows  $L_{\rho_2(\mathbb F_3)} \subset  L_{\rho_1(\mathbb F_3)}$ because of the uniqueness of the minimal compact invariant set. Since the actions on $L_{\rho_1(\mathbb F_3)}$ are equal, the orbits are dense and $L_{\rho_2(\mathbb F_3)} = L_{\rho_1(\mathbb F_3)}$.
\end{proof}
We will denote this set by $L=L_{\rho_1(\mathbb F_3)} = L_{\rho_2(\mathbb F_3)}$.

\begin{lemma} The representations $\rho_1$ and $\rho_2$ are semi conjugate by a map that is the identity on $L$. \end{lemma}
\begin{proof} We start by setting $h$ to be the identity on $L$. Let $I_1=\intoo{N}{S}$ be the connected component of $\Ss^1\setminus L$ stabilised by $\delta_1$.\\
\indent Set $h:I_1\to I_1$ to be a semi conjugacy between $\rho_1(\delta_1)$ and $\rho_2(\delta_1)$, and extend $h$ on images $\rho_1(\gamma)(I_1)$ for $\gamma \in \mathbb F_3$ by $h_{/\rho_1(\gamma)(I_1)} = \rho_2(\gamma) \circ h_{/I_1} \circ \rho_1(\gamma^{-1})$.\\
\indent Finally, set $h$ to be the identity on the other connected components of $\Ss^1\setminus L$. It provides a semi conjugacy between $\rho_1$ and $\rho_2$ that is the identity on $L$.
\end{proof}

Given $h:\Ss^1\to \Ss^1$  non increasing of degree one such that $h$ is the identity on $L$ and such that $(\rho_1,\rho_2,h)$ is a semi conjugate triple, we wish to show that it is area preserving.\\
\indent Since semi conjugacy is an equivalence relation, we can also consider $h^*:\Ss^1\to \Ss^1$ non decreasing of degree one such that $\rho_1\circ h^* = h^*\circ \rho_2$. Since we can construct it in the same way as for $h$, we can assume that $h^*$ is the identity on $L$.

\begin{rem} Such  a construction is not possible if we start with a Fuchsian representation of the fundamental group of a compact surface, because the limit set is the whole circle. This implies that we have to work with a surface of finite type, i.e. a closed surface of genus $g\ge 0$ with $r$ omitted points, so that $2g+r > 2$ (to ensure the existence of a hyperbolic structure, and that the representation in $\PSL(2,\R)$ is non elementary). In the construction of $\rho_2$, we used two technical properties of $\rho_1$. First, it is important to work with a free group, so that a choice of images of the generators in $\Diff(\Ss^1)$ always corresponds to a representation of the fundamental group considered. This is always the case with a non compact surface of  finite type (i.e. $r\ge 1$). The second important condition is that we can choose a system of generators so that one of the generators represents a simple loop around one of the omitted points. This is not possible with a simple punctured torus (where the fundamental group is $\mathbb F_2=\langle a,b\rangle$ and the simple loop around the omitted point is the commutator $[a,b]=a^{-1}b^{-1}ab$). This is why we used the twice punctured torus, but the sphere with three omitted points (i.e. a pair of pants) could also have been used. More generally, this construction works for a sphere with three or more removed points, or a closed surface of genus $g\ge 2$ with two or more removed points.
\end{rem}

\subsection{The  flow associated to $(\rho_1,\rho_2)$}We consider the following three manifold: $$\Sigma_h=\{(a,b,c)\in (\Ss^1)^3 \vert a<b<z<a ~\forall z\in h^{-1}(\{c\})\}$$ \indent  The group $\mathbb F_3$ acts on $\Sigma_h$  by $\gamma.(a,b,c)=(\rho_1(\gamma)(a),\rho_1(\gamma)(b),\rho_2(\gamma)(c ))$.
\begin{prop} The action of $\mathbb F_3$ on $\Sigma_h$ is properly discontinuous.
\end{prop}
\begin{proof} Assume that $\gamma_k\to \infty$ and that there is a sequence $(a_k,b_k,c_k)\to (a,b,c)\in \Sigma_h$ such that $\gamma_k.(a_k,b_k,c_k)\to (u,v,w)\in \Sigma_h$. Up to a subsequence and up to replacing $\gamma_k$ with $\gamma_k^{-1}$, we can assume that $\rho_1(\gamma_k)(x)\to v$ for all $x\ne u$. \\
\indent Assume that $a$ does not bound an interval where $\rho_1\ne \rho_2$. In that case, one can find a compact interval $K\subset \Ss^1$ bounded by points of $L$ such that $a\notin K$ and $c\in \mathring K$. The sequence of intervals $\rho_1(\gamma_k)(K)$ collapses to the point $\{v\}$. Since $\rho_2(\gamma_k)(K)=\rho_1(\gamma_k)(K)$ and $c_k\in K$, we see that  $\rho_2(\gamma_k)(c_k)\to v$, hence $w=v=h(v)$, which is absurd because $(u,v,w)\in \Sigma_h$.\\
\indent We now know that $u=N$ where $I=\intoo{N}{S}$ is an interval bounded by fixed points for a conjugate $\delta$ of $\delta_1$, such that the third fixed point $P$ is in $I$. If $c_k$ were not in $I$ for $k$ large enough, then we could still find a compact interval $K$ as above, which is impossible, hence $c_k\in I$ for $k$ large enough. This implies that $\gamma_k$ stabilizes $I$, which gives us $v=S$ and $w=P$. This is also impossible because $(N,S,P)\notin \Sigma_h$ by construction of $h$.
 \end{proof}
 We can now consider the quotient manifold $N_h=\Sigma_h/\mathbb F_3$.

\subsection{Invariant volume on the hyperbolic set} The projection on $N_h$ of the constant vector field $(0,1,0)$ on $\Sigma_h$  can be reparametrised into a smooth flow $\psi^t$. Consider the map $\tilde H^* : \Sigma_h \to \Sigma_3$ defined by $\tilde H^*(x,y,z) = (x,y,h^*(z))$. It induces a map $H^*:N_h\to M=\rm{T}^1\h^2/\rho_1(\mathbb F_3)$.  Its restriction to $\Omega_\psi$ is a diffeomorphism onto $\Omega_\p$ that sends $\psi^t$ to a reparametrisation of $\p^t$. From this we deduce that $\Omega_{\psi}$ is a compact hyperbolic set for $\psi^t$. If the image $x\in N_h$ of $(x_-,x_0,x_+)\in \Sigma_h$ is in $\Omega_{\psi}$, then the stable (resp. unstable) manifold of $x$ is the set of images of points $(y_-,y_0,y_+)$ such that $y_+=x_+$ (resp. $y_-=x_-$).\\
\indent We will use this flow in order to extend the volume form to the stable and unstable manifolds of the non wandering set.

\begin{lemma} There is a continuous volume form $\omega_1$ on $N_h$ that is invariant under $\psi^t$ at points of $W^s(\Omega_{\psi})\cup W^u(\Omega_{\psi})$. \end{lemma}

\begin{proof}  The differentiable conjugacy on the non wandering set  implies that there is a smooth volume form $\omega_0$ on $N_h$ that is preserved by the flow at points of the non wandering set. Hence, if $\psi^{t*}\omega_0=e^{-A(t,x)}\omega_0$ and $\alpha(x)=\frac{\partial A}{\partial t}(0,x)$, then $\alpha =0$ on $\Omega_{\psi}$. We will now construct a smooth function $\sigma$ on $N_h$ such that $\sigma(\psi^t(x))-\sigma(x)=\int_0^t\alpha(\psi^s(x))ds$ for all $x\in W^s(\Omega_{\psi})\cup W^u(\Omega_{\psi})$, so that $\omega_1=e^\sigma \omega_0$ meets our requirements.\\ 
\indent If $x\in W^s(z)$ with $z\in \Omega_{\psi}$, and if we have found such a function $\sigma$, then $\sigma(\psi^t(x))\approx \sigma(\psi^t(z))=0$ for $t$ large enough, hence $\sigma(x)=-\int_0^{\infty}\alpha(\psi^t(x))dt$. We will use this formula as a definition of $\sigma$. If it is well defined, then it satisfies the cohomology equation. \\
\indent Let $C>0$ be such that $d(\psi^t(x),\psi^t(z))\le Ce^{-t}$ (locally $C$ can be chosen independently from $x$ and $z$). Let $k$ be a Lipschitz constant for $\alpha$ in a neighbourhood $U$ of $\Omega_{\psi}$. For $t$ such that $\psi^t(x)\in U$ (which is locally uniform in $x$), we have: $$\vert \alpha(\psi^t(x))\vert \le \underbrace{\vert \alpha(\psi^t(z))\vert}_{=0} + k\underbrace{d(\psi^t(x),\psi^t(z))}_{\le Ce^{-t}}$$
\indent This gives us uniform convergence, hence $\sigma$ is well defined and continuous. By applying the same reasoning with negative times, we define $\sigma$ on $W^u(\Omega_{\psi})$.
\end{proof}

\subsection{Going back from the flow to $M_h$} Now that we have found an invariant volume form on a larger set for the  flow $\psi^t$, we need to translate it in terms of the action on $M_h$.
\begin{lemma} If there is a continuous volume form $v$ on $N_h$ preserved by $\psi^t$ at points of $W^s(\Omega_{\psi})\cup W^u(\Omega_{\psi})$, then there is a continuous volume form $\omega$ on $M_h$ preserved by $(\rho_1,\rho_2)$  at points of $L\times \Ss^1\cup \Ss^1\times L$. \end{lemma}

\begin{proof}  Let $\omega_1=e^{\sigma}\omega_0$ be a continuous volume form on $N_h$ that is invariant at points of $W^s(\Omega_{\psi})\cup W^u(\Omega_{\psi})$. Let $\tilde \omega_1$ be its lift to $\Sigma_h$ and write: $$\tilde \omega_1=\tilde \omega_1(x_-,x_0,x_+)dx_-\wedge dx_0\wedge dx_+$$ \indent  If $x_-$ or $x_+$ is in $L$, then the image in $N_h$ is in $W^s(\Omega_{\psi})\cup W^u(\Omega_{\psi})$, and the invariance under the  flow $\psi^t$ gives us $\tilde \omega_1(x_-,x_0,x_+) = \tilde \omega_1(x_-,x'_0,x_+)$ for all $x'_0$ such that $(x_-,x'_0,x_+)\in \Sigma_h$. \\
\indent Choose a continuous map $i_0:M_h\to \Ss^1$ such that $(x_-,i_0(x_-,x_+),x_+)\in \Sigma_h$ for all $(x_-,x_+)\in M_h$, and let $\omega_2(x_-,x_+)=\tilde \omega_1(x_-,i_0(x_-,x_+),x_+)$ for $(x_-,x_+)\in M_h$. If $x_-$ or $x_+$ is in $L$ and $\gamma\in \mathbb F_3$, then the invariance under $\psi^t$ gives us: 
\begin{eqnarray*} &~& \omega_2(\rho_1(\gamma)(x_-),\rho_2(\gamma)(x_+)) \rho_1(\gamma)'(x_-)\rho_2(\gamma)'(x_+) \\&=& \tilde \omega_1 ( \rho_1(\gamma)(x_-), i_0(\rho_1(\gamma)(x_-),\rho_2(\gamma)(x_+)),\rho_2(\gamma)(x_+)) \rho_1(\gamma)'(x_-)\rho_2(\gamma)'(x_+)\\  &=& \tilde \omega_1 ( \rho_1(\gamma)(x_-), \rho_1(\gamma)(i_0(x_-,x_+)),\rho_2(\gamma)(x_+)) \rho_1(\gamma)'(x_-)\rho_2(\gamma)'(x_+)\\ &=& \tilde \omega_1(x_-,i_0(x_-,x_+),x_+) \\&=& \omega_2(x_-,x_+) \end{eqnarray*}
\indent We have defined a continuous volume form $\omega_2$ on $M_h$ that is $(\rho_1,\rho_2)$-invariant at points of $(L\times \Ss^1\cup \Ss^1\times L)\cap M_h$.
\end{proof}

\subsection{Horizontal strips}The first step in extending $\omega$ to all of $M_h$ is to extend it to horizontal strips delimited by elements of $L$, so that we only need to deal with invariance under one element of the group.

\begin{lemma} \label{vertical} Let $I$ be a connected component of $\Ss^1\setminus L$, and let $\gamma\in \mathbb F_3$ be a generator of its stabilizer. There is a continuous volume form $\omega$ on $\Ss^1\times \overline I \setminus G(h)$ that is invariant by $(\rho_1(\gamma),\rho_2(\gamma))$ and that is equal to $\omega_2$ on $L\times \Ss^1\cup \Ss^1\times L$. \end{lemma} 
 

\begin{proof} If $\gamma$ is conjugate to $\delta_1$, then Proposition \ref{elementary_non_convergence}, states that there is  a continuous volume form $\omega_{\gamma}$ on $M_h$ that is invariant under $(\rho_1(\gamma),\rho_2(\gamma))$. If $\gamma$ is conjugate to $\delta_2$, then Proposition 1.7 of \cite{Monclaira} gives the same result (the proof is almost identical to Proposition \ref{elementary_non_convergence}).\\ 
\indent Let $a\in L\setminus \overline I$. The interval $\intfo{a}{\rho_1(\gamma)(a)}$ is a fondamental domain for the action of $\rho_1(\gamma)$ on $\Ss^1 \setminus \overline I$, i.e. for every $y\in \Ss^1 \setminus \overline I$ there is a unique $n_y\in \Z$ such that $\rho_1(\gamma^{n_y})(y)\in \intfo{a}{\rho_1(\gamma)(a)}$. We set $\omega=\omega_2$ on $\intfo{a}{\rho_1(\gamma)(a)}\times \overline I$ and extend $\omega$ to  $  (\Ss^1\setminus \overline I)\times \overline I$ by using the equivariance formula: $$\frac{\omega(x,y)}{\omega_2(\rho_1(\gamma^{n_y})(x),\rho_2(\gamma^{n_y})(y))}= \rho_1(\gamma^{n_y})'(x) \rho_2(\gamma^{n_y})'(y)$$
\indent We have to show that $\omega$ is continuous. First, remark that it is continuous at every point of  $ \intfo{a}{\rho_1(\gamma)(a)}\times \overline I$: if $(x_n,y_n)\to (a,y)$ with $\rho_1(\gamma)(x_n)\in \intfo{a}{\rho_1(\gamma)(a)}$, then using the fact that  $a\in L$ and because the  volume $\omega_2$ is preserved at $(a,y)$, we get: \begin{eqnarray*} \omega(x_n,y_n)&=&\omega_2(\rho_1(\gamma)(x_n),\rho_1(\gamma)(y_n))\rho_1(\gamma)'(x_n)\rho_1(\gamma)'(y_n)\\ &\to & \omega_2(\rho_1(\gamma)(a),\rho_1(\gamma)(y))\rho_1(\gamma)'(a)\rho_1(\gamma)'(y)\\ &  ~& ~~~~~~=\omega_2(a,y)=\omega(a,y) \end{eqnarray*}
\indent If $(x_k,y_k)\in I\times \overline I \to (x,y)\in M_h$ with $x \in \partial I$, then set $n_k=n_{x_k}$, as well as $u_k=\rho_1(\gamma^{n_k})(x_k)$ and $v_k=\rho_1(\gamma^{n_k})(y_k)$. By definition, we have: $$\omega(x_k,y_k)=\omega_2(u_k,v_k)\rho_1(\gamma^{n_k})'(x_k)\rho_1(\gamma^{n_k})'(y_k)$$
\indent Since $\omega_{\gamma}$ is invariant under $\rho_1(\gamma)$, we have: $$\rho_1(\gamma^{n_k})'(x_k)\rho_1(\gamma^{n_k})'(y_k) = \frac{\omega_{\gamma}(x_k,y_k)}{\omega_{\gamma}(u_k,v_k)}$$
\indent These two equalities give us: $$\omega(x_k,y_k) =\frac{\omega_2(u_k,v_k)}{\omega_{\gamma}(u_k,v_k)}\omega_{\gamma}(x_k,y_k)$$
\indent The continuity of $\omega_{\gamma}$ gives us $\omega_{\gamma}(x_k,y_k)\to \omega_{\gamma}(x,y)$.\\
\indent Since $x_k\to x\in \partial I$, we have $n_k\to \infty$ and $v_k\to v$ where $v$ is the other extremal point of $I$. By using the uniform continuity of $\omega_2$ and $\omega_{\gamma}$ on $ \intff{a}{\rho_1(\gamma)(a)}\times \overline I$, we obtain:
$$\omega(x_k,y_k) \sim \frac{\omega_2(u_k,v)}{\omega_{\gamma}(u_k,v)} \omega_{\gamma}(x,y)$$

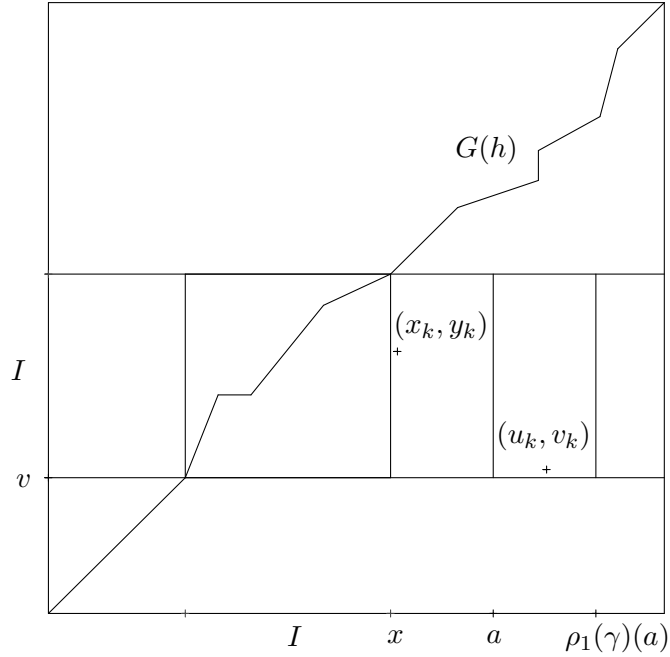
\begin{figure}[h]
\definecolor{uuuuuu}{rgb}{0.27,0.27,0.27}
\begin{tikzpicture}[line cap=round,line join=round,>=triangle 45,x=1.0cm,y=1.0cm,scale=0.9]
\clip(-3.5,-4.8) rectangle (16.9,5.4);
\draw (-1,5)-- (8,5);
\draw (-1,-4)-- (-1,5);
\draw (-1,-4)-- (8,-4);
\draw (8,-4)-- (8,5);
\draw (1,1)-- (4,1);
\draw (4,-2)-- (1,-2);
\draw (-1.7,-0.12) node[anchor=north west] {$I$};
\draw (5.25,-4.1) node[anchor=north west] {$a$};
\draw (6.4,-4) node[anchor=north west] {$\rho_1(\gamma)(a)$};
\draw (5.4,-1.02) node[anchor=north west] {$(u_k,v_k)$};
\draw (3.9,0.6) node[anchor=north west] {$(x_k,y_k)$};
\draw (2.34,-4.04) node[anchor=north west] {$I$};
\draw (8,-2)-- (-1,-2);
\draw (-1,1)-- (8,1);
\draw (5.5,-2)-- (5.5,1);
\draw (7,-2)-- (7,1);
\draw (1,1)-- (1,-2);
\draw (4,1)-- (4,-2);
\draw (-1,-4)-- (1,-2);
\draw (1,-2)-- (1.48,-0.78);
\draw (1.48,-0.78)-- (1.96,-0.78);
\draw (1.96,-0.78)-- (3.02,0.54);
\draw (3.02,0.54)-- (4,1);
\draw (4,1)-- (4.98,1.98);
\draw (4.98,1.98)-- (6.16,2.38);
\draw (6.16,2.38)-- (6.16,2.82);
\draw (6.16,2.82)-- (7.06,3.32);
\draw (7.06,3.32)-- (7.32,4.32);
\draw (7.32,4.32)-- (8,5);
\draw (4.8,3.2) node[anchor=north west] {$G(h)$};
\draw (3.8,-4.1) node[anchor=north west] {$x$};
\draw (-1.62,-1.8) node[anchor=north west] {$v$};
\begin{scriptsize}
\draw [color=black] (-1,-2)-- ++(-1.5pt,0 pt) -- ++(3.0pt,0 pt) ++(-1.5pt,-1.5pt) -- ++(0 pt,3.0pt);
\draw [color=black] (4.1,-0.14)-- ++(-1.5pt,0 pt) -- ++(3.0pt,0 pt) ++(-1.5pt,-1.5pt) -- ++(0 pt,3.0pt);
\draw [color=black] (6.28,-1.88)-- ++(-1.5pt,0 pt) -- ++(3.0pt,0 pt) ++(-1.5pt,-1.5pt) -- ++(0 pt,3.0pt);
\draw [color=black] (5.5,-4)-- ++(-1.5pt,0 pt) -- ++(3.0pt,0 pt) ++(-1.5pt,-1.5pt) -- ++(0 pt,3.0pt);
\draw [color=black] (7,-4)-- ++(-1.5pt,0 pt) -- ++(3.0pt,0 pt) ++(-1.5pt,-1.5pt) -- ++(0 pt,3.0pt);
\draw [color=uuuuuu] (1,-4)-- ++(-1.5pt,0 pt) -- ++(3.0pt,0 pt) ++(-1.5pt,-1.5pt) -- ++(0 pt,3.0pt);
\draw [color=uuuuuu] (-1,1)-- ++(-1.5pt,0 pt) -- ++(3.0pt,0 pt) ++(-1.5pt,-1.5pt) -- ++(0 pt,3.0pt);
\draw [color=uuuuuu] (4,-4)-- ++(-1.5pt,0 pt) -- ++(3.0pt,0 pt) ++(-1.5pt,-1.5pt) -- ++(0 pt,3.0pt);
\end{scriptsize}
\end{tikzpicture}\caption{Defining $\omega$ on horizontal strips} \label{fig:horizontal}
\end{figure}

\indent We now only have to deal with the restrictions of $\omega_2$ and $\omega_{\gamma}$  to the axes $\{x\} \times \Ss^1\cup \Ss^1\times \{v\}$ (see Figure \ref{fig:horizontal}), where continuous volume forms invariant under $(\rho_1(\gamma),\rho_2(\gamma))$ are unique up to multiplication by a constant: there is $\lambda>0$ such that $\omega_2(s,t)=\lambda \omega_{\gamma}(s,t)$ whenever $s=x$ or $t=v$. We can finally conlude: $$\omega(x_k,y_k) \to \lambda \omega_{\gamma}(x,y) = \omega_2(x,y)=\omega(x,y)$$
\indent Finally, we can extend $\omega$ to $ \Ss^1 \times \overline I  \setminus G(h)$ by setting $\omega = \lambda \omega_\gamma$ on $\overline I\times \overline I \setminus G(h)$.
\end{proof}

\subsection{Extending to $M_h$} \label{subsec:extending} We can now extend $\omega$ to $M_h$. Getting an invariant volume form is not complicated, however its regularity requires some work.\\
\indent Our proof of the regularity of $\omega$ on horizontal strips relied on the existence of a continuous invariant form by any element of $\mathbb F_3$. To deal with the invariance under the whole group, we will need a different method.
\begin{prop} \label{continuous} There is a continuous invariant form $\omega$ on $M_h$ that is invariant under $(\rho_1,\rho_2)$ and that is equal to $\omega_2$ on $L \times\Ss^1\cup\Ss^1\times L$. \end{prop}

\begin{proof} The action of $\mathbb F_3$ on the set of connected components of $\Ss^1\setminus L$ has two orbits. Let $I_1, I_2$ be the components preserved by $\delta_1$ and $\delta_2$. By Lemma \ref{vertical}, there is a continuous volume form $\omega$ on $\Ss^1\times \overline I_i \setminus G(h)$ that is equal to $\omega_2$ in restriction to $L\times \Ss^1\cup \Ss^1\times L$ and that is invariant under the stabilizer of $I_i$.  If $\gamma \in \mathbb F_3$, then we define $\omega$ on $\Ss^1\times \rho_2(\gamma) (\overline I_i) \setminus G(h)$ to be $(\rho_1(\gamma),\rho_2(\gamma))_*\omega$. This defines a volume form $\omega$ on $M_h$ that is $(\rho_1,\rho_2)$-invariant, continuous on all horizontal strips $\Ss^1\times \overline I \setminus G(h)$ where $I$ is a connected component of $\Ss^1\setminus L$ and equal to $\omega_2$ on $L \times \Ss^1\cup \Ss^1\cup L$.\\
\indent To show that $\omega$ is continuous, assume that $(x_k,y_k)\to (x,y)$ with $y\in L$ (if $y\notin L$, then there is a connected component $I$ of $\Ss^1\setminus L$ such that $y_k\in I$ for $k$ large enough, which gives us $\omega(x_k,y_k)\to \omega(x,y)$). If $y_k\in L$ for all $k$, then $\omega(x_k,y_k)=\omega_2(x_k,y_k)$ and we already have the continuity, hence we can assume that $y_k\notin L$ for all $k$. Up to considering two subsequences, we can assume that there is $\gamma_k\in \mathbb F_3$ such that $v_k=\rho_2(\gamma_k)(y_k) \in I_1$. By composing $\gamma_k$ with an element of the stabilizer of $I_1$, we can take $v_k$  in a compact interval $K\subset I_1$.\\
\indent Let $u_k=\rho_1(\gamma_k)(x_k)$. The definition of $\omega$ is:
$$\omega(x_k,y_k)=\omega(u_k,v_k)\rho_1(\gamma_k)'(x_k)\rho_2(\gamma_k)'(y_k)$$
\indent We have already seen that $\omega$ is continuous on $\Ss^1\times\overline I_1\setminus G(h)$ and $v_k\in I_1$. The problem in finding the limit of $\omega(x_k,y_k)$ is the control of the Jacobian product $\rho_1(\gamma_k)'(x_k)\rho_2(\gamma_k)'(y_k)$. However, we know that $\omega$ is continuous on $L\times \Ss^1\cup \Ss^1\times L$. We will use this fact to get rid of the derivatives: if $x'_k$ and $y'_k$ are sequences in $L$ such that $(x'_k, y'_k)\notin G(h)$, $(x'_k, y_k)\notin G(h)$ and $(x_k, y'_k)\notin G(h)$, then we set $u'_k=\rho_1(\gamma_k)(x'_k)$ and $v'_k=\rho_2(\gamma_k)(y'_k)$. The equivariance equation for $\omega$ gives us:  \begin{equation} \frac{\omega(x_k,y_k)}{\omega(x_k,y'_k)}\frac{\omega(x'_k,y'_k)}{\omega(x'_k,y_k)} = \frac{\omega(u_k,v_k)}{\omega(u_k,v'_k)}\frac{\omega(u'_k,v'_k)}{\omega(u'_k,v_k)} \label{equivariance2} \end{equation}
\indent We are now looking for suitable points $x'_k$ and $y'_k$. Let $I_1=\intoo{a}{b}$, and assume that $u_k$ does not admit $a$ as a limit point (up to considering two subsequences and replacing $a$ by $b$ in the following discussion, we can always assume that it is the case), i.e. that $a_k$ lies in a compact interval $J\subset \Ss^1\setminus\{a\}$. Let $v'_k=a$ and $y'_k=\rho_2(\gamma_k^{-1})(a) \to y$. If $x_k\in L$, then we choose $x'_k=x_k$. If $x_k\notin L$, then we set $x'_k$ to be an extremal point of the connected component of $\Ss^1\setminus L$ containing $x_k$, in a way such that $u'_k=\rho_1(\gamma_k)(x'_k)\in J$.\\
\indent We now have $y'_k\to y$ and $y_k\to y$, so we get:
$$\frac{\omega(x_k,y_k)}{\omega(x_k,y'_k)}\frac{\omega(x'_k,y'_k)}{\omega(x'_k,y_k)}  \sim \frac{\omega(x_k,y_k)}{\omega(x,y)}\frac{\omega(x'_k,y)}{\omega(x'_k,y)} = \frac{\omega(x_k,y_k)}{\omega(x,y)}$$
\indent We wish to show that this quantity converges to $1$ as $k\to \infty$. The compact set  $E=J\times \{b\} \cup  \Ss^1\setminus I_1 \times K$ of $M_h$  contains the sequences $(u_k,v_k)$,  $(u_k,v'_k)$, $(u'_k,v_k)$ and $(u'_k,v'_k)$. Consequently, the ratio \eqref{equivariance2} lies in a compact set of $\intoo{0}{+\infty}$, and it is enough to see that its only possible limit is $1$. If there is a subsequence such that the ratio \eqref{equivariance2} converges to $\lambda \in \intoo{0}{+\infty}$, then up to another subsequence, we can assume that the sequence $\gamma_k$ has the convergence property: there are $N,S\in \Ss^1$ such that $\rho_1(\gamma_k)(z)\to N$ for all $z\ne S$. Since $\rho_1(\gamma_k^{-1})(z)\to x$ for all $z\in I_1$, we see that $S$ in necessarily equal to $x$, hence the sequences $v_k$ and $v'_k$ converge to $N\in \Ss^1$. We get: $$ \frac{\omega(u_k,v_k)}{\omega(u_k,v'_k)}\frac{\omega(u'_k,v'_k)}{\omega(u'_k,v_k)} \to \frac{\omega(u,N)}{\omega(u,N)}\frac{\omega(a,N)}{\omega(a,N)} = 1   $$ 
\indent This shows that $\lambda =1$, therefore $\omega(x_k,y_k)\to \omega(x,y)$ and $\omega$ is continuous.
\end{proof}

We only showed that $\omega$ is continuous, but following \cite{Monclaira} one can show that it is possible to obtain smooth volume form.

\subsection*{Acknowledgments} This work corresponds to Chapter 5 and section 3 of Chapter 3 in my PhD thesis \cite{these}.   I would like to thank my  advisor Abdelghani Zeghib for his help throughout this work, as well as Andrés Navas for sharing his knowledge about Denjoy diffeomorphisms.

\footnotesize \textsc{UMPA, \'Ecole Normale Supérieure de Lyon, 46 allée d'Italie, 69364 Lyon Cedex 07, France}\\
 \emph{E-mail address:}  \verb|daniel.monclair@ens-lyon.fr|


\begin{thebibliography}{2}


\bibitem[Ahl64]{ahlfors} L.V. Ahlfors: \emph{Finitely generated Kleinian groups}, Amer. J. Math., \textbf{86} (1964), p. 413-429


\bibitem[Bar95]{B95} T. Barbot: \emph{Caractérisation des flots d'Anosov en dimension $3$ par leurs feuilletages faibles}, Ergodic Theory Dynam. Systems, \textbf{15} (1995), no. 2, p.247-270

\bibitem[Bar96]{Ba96} T. Barbot: \emph{Flots d'Anosov sur les variétés graphées au sens de Waldhausen}, Ann. Inst. Fourier, \textbf{46} (1996), p. 1451-1517

\bibitem[Bar01]{B01} T. Barbot: \emph{Plane affine Geometry of Anosov flows},
Annales Scientifiques de l'Ecole Normale Supérieure, \textbf{34} (2001), no. 6, p. 871-889

\bibitem[Ber65]{bers} L. Bers: \emph{Automorphic forms and Poincaré series for infinitely generated Fuchsian groups}, Amer. J. Math., \textbf{87} (1965), p. 196-214

\bibitem[But00]{Button} J. Button: \emph{Matrix representations and the Teichm\"uller space of the twice punctured torus}, Conformal Geometry and Dynamics, \textbf{4} (2000),  p. 97-107

\bibitem[CJ94]{CJ} A. Casson, D. Jungreis: \emph{Convergence groups and Seifert fibered $3$-manifolds}, Invent. Math., \textbf{118} (1994), no. 3, p.441-456

\bibitem[Den32]{Denjoy} A. Denjoy: \emph{Sur les courbes définies par les équations différentielles à la surface du tore}, J. Math. Pures Appl., (9) \textbf{11} (1932), p. 333-375

\bibitem[FH13]{FH} P.~Foulon, B.~Hasselblatt: \emph{Contact Anosov flows on hyperbolic $3$-manifolds}, Geometry \& Topology, \textbf{17} (2013), p. 1225-1252

\bibitem[Gab92]{Gabai} D.~Gabai: \emph{Convergence groups are Fuchsian groups}, Ann. of Math. \textbf{136} (1992), p. 447-510

\bibitem[Ghy87a]{Gh87a} E.~Ghys: \emph{Flots d'Anosov dont les feuilletages stables sont différentiables}, Ann. Scient. Ec. Norm. Sup.,  (4) \textbf{20} (1987), no. 2, p. 251-270

\bibitem[Ghy87b]{Gh87b} E.~Ghys: \emph{Groupes d'homéomorphismes du cercle et cohomologie bornée},  Contemporary Mathematics,  \textbf{58} (1987), Part III, p. 81-106

\bibitem[Ghy92]{Gh92} E.~Ghys: \emph{Déformations de flots d'Anosov et de groupes fuchsiens}, Ann.  Inst. Fourier, \textbf{42} (1992), p. 209-247

\bibitem[Ghy93]{Gh93} E.~Ghys: \emph{Rigidité différentiable des groupes fuchsiens}, Publ. Math. de l'I.H.E.S., \textbf{78} (1993), p. 163-185

\bibitem[Ghy01]{Gh01} E.~Ghys: \emph{Groups acting on the circle}, Enseign. Math., (2) \textbf{47} (2001), no. 3-4, p.329-407

\bibitem[Her79]{Herman} M.R.~Herman: \emph{Sur la conjugaison différentiable des difféomorphismes du cercle à des rotations}, Publ. Math. de l'I.H.E.S., \textbf{49} (1979), p. 5-234

\bibitem[HK90]{HK} S.~Hurder, A.~Katok: \emph{Differentiability, rigidity, and Godbillon-Vey classes for Anosov flows}, Publ. Math. de l’I.H.E.S.,  \textbf{72} (1990), p. 5-61

\bibitem[HP69]{HP} M. Hirsh, C. Pugh : \emph{Stable manifolds for hyperbolic sets}, Bull. Amer. Math. Soc., \textbf{75}, no. 1 (1969), p. 149-152

\bibitem[KH95]{KH} A. Katok, B.  Hasselblatt: \emph{Introduction to the Modern Theory of Dynamical Systems}, Encyclopedia of Mathematics and its Applications, \textbf{54}, Cambridge University Press, 1995

\bibitem[Liv71]{livsic} A. N. Livšic: \emph{Homology properties of U systems}, Math. Notes, \textbf{10} (1971), p. 758-763

\bibitem[Mat87]{Matsumoto} S. Matsumoto: \emph{Some remarks on foliated $S^1$ bundles}, Invent. Math., \textbf{90} (1987), p. 343-358

\bibitem[Mon14a]{Monclaira} D. Monclair: \emph{Differential conjugacy for groups of area preserving circle diffeomorphisms}, arXiv:1402.0424

\bibitem[Mon14b]{Monclairb} D. Monclair: \emph{Convergence groups and semi conjugacy}, arXiv:1402.7179

\bibitem[Mon14c]{these} D. Monclair, \emph{Dynamique lorentzienne et groupes de difféomorphismes du cercle}, PhD Thesis, 2014


\bibitem[Nav11]{navas_book} A. Navas: \emph{Groups of circle diffeomorphisms}, Chicago Lectures in Mathematics (2011)









\end{thebibliography}
\end{document}